\documentclass[12pt,reqno]{amsart}
\usepackage{amsmath}
\usepackage{amsfonts}
\usepackage{amssymb, amsthm, epsfig}
\usepackage{mathrsfs}
\usepackage{refcheck}
\usepackage{hyperref}

\usepackage[utf8]{inputenc}
\usepackage[T1]{fontenc}
\usepackage{geometry}
\usepackage{graphicx}
\usepackage{verbatim}
\usepackage{color}
\usepackage{listings}
\usepackage{listings}
\usepackage{xcolor}
\theoremstyle{plain}
\newtheorem{thm}{Theorem}[section]

\newtheorem{prop}[thm]{Proposition}

\theoremstyle{definition}
\newtheorem{defn}{Definition}[section]
\theoremstyle{remark}
\newtheorem{rem}{Remark}[section]

\numberwithin{equation}{subsection}

\allowdisplaybreaks

\usepackage{lineno}  
\newcommand{\norm}[1]{\left\Vert#1\right\Vert}
\newcommand{\abs}[1]{\left\vert#1\right\vert}
\newcommand{\set}[1]{\left\{#1\right\}}

\newcommand{\qnt}[1]{\left(#1\right)}

\newcommand{\me}{\mathrm{e}}
\newcommand{\dif}{\mathrm{d}}

\DeclareSymbolFont{lettersA}{U}{pxmia}{m}{it}
\DeclareMathSymbol{\piup}{\mathord}{lettersA}{"19}

\makeatletter

\newcommand{\Rmnum}[1]{\expandafter\@slowromancap\romannumeral#1@}
\makeatother

\newcommand{\G}{\mathcal{G}}

\begin{document}
\title[Approximate transonic shock solution for hypersonic flow past a large curved convex wedge]
{Approximate transonic shock solution for hypersonic flow past a large curved convex wedge}
\author{Dian Hu}

\address{School of Mathematics, East China University Of Science And Technology,
Shanghai, 200237, China.}
\email{\tt hudianaug@qq.com}



\date{\today}

\begin{abstract}
In this paper, we study the existence of a global transonic shock generated by hypersonic potential flow over a large curved convex wedge. Modeling 2-dimensional steady potential flow leads to a free boundary value problem of quasilinear equation. By applying the hodograph transformation, in which one of the coordinate is the unknown function, this problem is reduced to a free boundary value problem of linear equation. For hypersonic incoming flow with adiabatic index $\gamma=2$, we construct an approximate boundary based on the asymptotic state. Then, within the region defined by this approximate boundary, we solve a narrow-region elliptic boundary value problem. This solution serves as the approximate solution to the free boundary problem that we seek. Utilizing uniform weighted Schauder estimates for the elliptic mixed boundary value problem in the narrow region, we obtain an error estimate for the free boundary.
\end{abstract}

\maketitle
\tableofcontents

\section{Introduction}\label{sect:Introduction}
We consider the potential flow
\begin{equation}\label{eq:PotentialEquations}
\begin{cases}
(\rho u)_x + (\rho v)_y = 0,\\
v_x-u_y=0,
\end{cases}
\end{equation}
where the first line is the conservation law of mass and the second is the irrotational condition. Here $\rho$ and $(u, v)$ denote respectively the density of the mass and the velocity of the flow. For polytropic gas, the pressure of the flow is
\begin{equation*}
p(\rho)=A\rho^\gamma,
\end{equation*}
where $A>0$ is a constant and $\gamma\in(1, 3)$ the adiabatic index of the gas. In this paper, we assume $A=1$. The sound speed $c(\rho)$ is defined by
\begin{equation*}
c^2(\rho)=\frac{\dif p}{\dif \rho}=\gamma\rho^{\gamma-1}.
\end{equation*}
The density $\rho$ and the velocity $(u, v)$ are  connected by the Bernoulli’s law
\begin{equation}\label{eq:BernoulliLaw}
\frac{u^2+v^2}{2}+\frac{c^2}{\gamma-1}=\frac{u^2+v^2}{2}+\frac{\gamma\rho^{\gamma-1}}{\gamma-1}=\frac{\bar{q}^2}{2},
\end{equation}
where $\bar{q}$ is the limit speed. The Mach number of the flow is defined by
\begin{equation}\label{eq:MachNumber}
M=\frac{q}{c}
\end{equation}
with $q=\sqrt{u^2+v^2}$.

For the smooth flow fields, we can rewrite \eqref{eq:PotentialEquations} in following non divergent form
\begin{equation}\label{eq:EulerEquationsNC}
\begin{cases}
(c^2-u^2)u_x - uv(u_y+v_x)+(c^2-v^2)v_y=0,\\
u_y-v_x=0.
\end{cases}
\end{equation}
It can be written to the matrix form
\begin{equation}\label{eq:EulerEquationsM}
\begin{pmatrix}
u\\
v
\end{pmatrix}_x 
+ 
\begin{pmatrix}
\frac{-2uv}{c^2-u^2}& \frac{c^2-v^2}{c^2-u^2}\\
-1& 0
\end{pmatrix}
\begin{pmatrix}
u\\
v
\end{pmatrix}_y=0.
\end{equation}
For this quasilinear system, one can calculate that it is hyperbolic for $u^2+v^2>c^2$ and elliptic for $u^2+v^2<c^2$. The flow field is called supersonic flow and subsonic flow respectively.

On the shock $\mathsf{S}: x=\phi(y)$, we have the following Rankine-Hugoniot jump condition (which we abbreviate as R-H conditions
hereafter)
\begin{equation}\label{eq:RHCondtionb}
\begin{cases}
\phi'=\frac{[\rho u]}{[\rho v]};\\
\phi'=-\frac{[v]}{[u]},
\end{cases}
\end{equation}
where $[F]$ in this paper is defined by
\begin{equation*}
[F](\phi(y), y)=\lim\limits_{x\rightarrow\phi(y)+0}F(x, y)-\lim\limits_{x\rightarrow\phi(y)-0}F(x, y).
\end{equation*}

By symmetry, the flow needs only to be considered in the upper half-plane. A convex wedge $\mathsf{W}$ in this plane has boundary
\begin{equation}\label{eq:WallPositionCondition}
	x = b(y).
\end{equation}
It follows the slip wall condition
\begin{equation}\label{eq:WallCondition}
u-vb'(y)=0.
\end{equation}
In this paper, we assume that $b\in C^\infty$ satisfies
\begin{equation}\label{eq:WallSlopeRange}
b':[0, +\infty)\mapsto[\flat, \sharp),
\end{equation}
\begin{equation}\label{eq:SupersonicWallCondition}
0<\flat\leq b'<\sharp<\sqrt{\frac{\gamma-1}{2}},
\end{equation}
and
\begin{equation}\label{eq:WallCondition2nd}
b''(y)>0.
\end{equation}

For an incoming supersonic flow with velocity $(\underline{u}, 0)$ and presented wall $\mathsf{W}$, we intend to determine the shock position $\mathsf{S}$ and solve out $(u, v)$ between $\mathsf{S}$ and $\mathsf{W}$ with the shock free boundary value condition \eqref{eq:RHCondtionb} and wall condition \eqref{eq:WallCondition}. Especially, the entropy condition $\underline{u}>q$ holds along the shock. In the following, this problem is abbreviated as $\mathbf{(P)}$.

Eliminating $\phi'$ in \eqref{eq:RHCondtionb}, we can deduce an equation along $\mathsf{S}$ as follows
\begin{equation}\label{eq:shockpolar}
[\rho u][u]+[\rho v][v]=(\rho u-\underline{\rho} \underline{u})(u-\underline{u})+\rho v^2=0.
\end{equation}
For any incoming supersonic flow with velocity $(\underline{u}, 0)$ and density $\underline{\rho}$, the downstream states $(u, v)$ satisfying the governing equations along $\mathsf{S}$ form a curve in the $(u, v)$-plane, known as the shock polar.

When $\mathsf{W}$ is a straight wedge ($x = b_0y$, $b_0 > 0$), the flow in the region between shock $\mathsf{S}$ and wall $\mathsf{W}$ may remain constant. The problem $\mathbf{(P)}$ can be analysed via the shock polar \eqref{eq:shockpolar}. As illustrated in Figure \ref{fig:ShockPolar}, for the given velocity of the incoming flow $(\underline{u}, 0)$, the velocity $(u, v)$ of the down stream must locate on the shock polar branch $u<\underline{u}$. By the wall condition \eqref{eq:WallCondition}, the down stream states $(u, v)$ is the intersection of the shock polar and the line $u=b_0v$. Meanwhile, the slope of the shock can be determined by the second equation of \eqref{eq:RHCondtionb}. Geometrically, the direction of the shock wave is perpendicular to the line connecting $(u, v)$ and $(\underline{u}, 0)$. Thus, for the straight wall case, a piecewise constant solution can be constructed. Based on the state of the flow field behind the shock, the waves can be classified into two categories. One is the supersonic shock, where the flow field behind the shock is supersonic. The other is the transonic shock, where the flow field behind the shock is subsonic. For curved wall boundaries, nonlinear perturbation stability is analysed in \cite{ChenFang2007JDE,YinZhou2009JDE}, proving the mathematical stability of the piecewise constant wave structure. However, highly curved walls $\mathsf{W}$ introduce significant complexity. A particularly challenging case is the blunt body ($b'(0) = 0$), where a detached transonic shock forms upstream. As this shock adjusts to the obstacle's curvature, its strength diminishes and it transitions to a supersonic shock (Figure \ref{fig:bluntbody}). This problem presents multifaceted complexity: the post-shock flow field (to be determined) contains both supersonic and subsonic regions with an a priori unknown sonic transition/degeneration line; the shock structure is globally determined by the coupled mixed-type flow field behind it; substantial flow deflection occurs across the shock; detachment manifests; among other complex features.
\begin{figure}[htbp]
	\setlength{\unitlength}{1bp}
	\begin{center}
		\begin{picture}(190,190)
		\put(-60,-10){\includegraphics[scale=0.4]{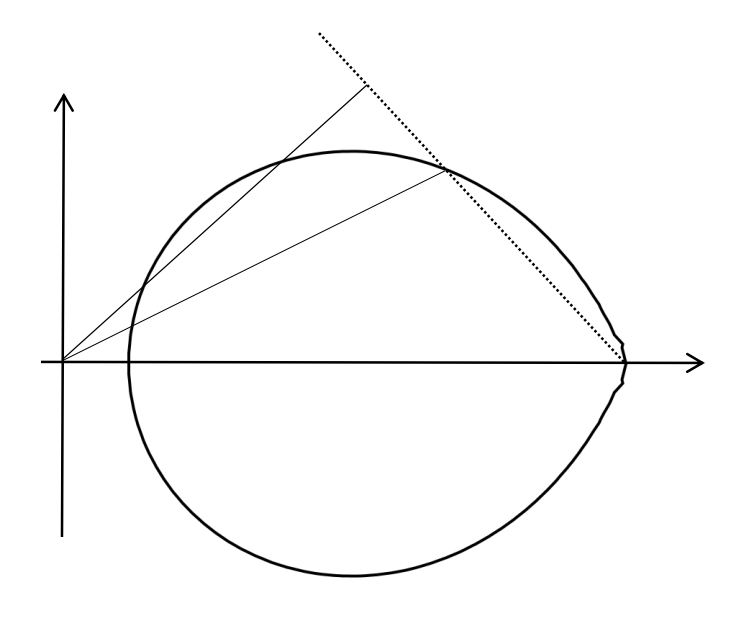}}
		\put(-52,145){\small$v$}
		\put(143,77){\small$u$}
		\put(128,63){\tiny$(\underline{u}, 0)$}
		\put(76,131){\tiny$(u, v)$}
		\put(53,155){\small$\mathsf{S}$}
		\put(-55,60){\small$O$}
		\end{picture}
	\end{center}
	\caption{shock polar for $\underline{u}<\bar{q}$}\label{fig:ShockPolar}
\end{figure}
\begin{figure}[htbp]
	\setlength{\unitlength}{1bp}
	\begin{center}
		\begin{picture}(190,190)
		\put(-60,-10){\includegraphics[scale=0.4]{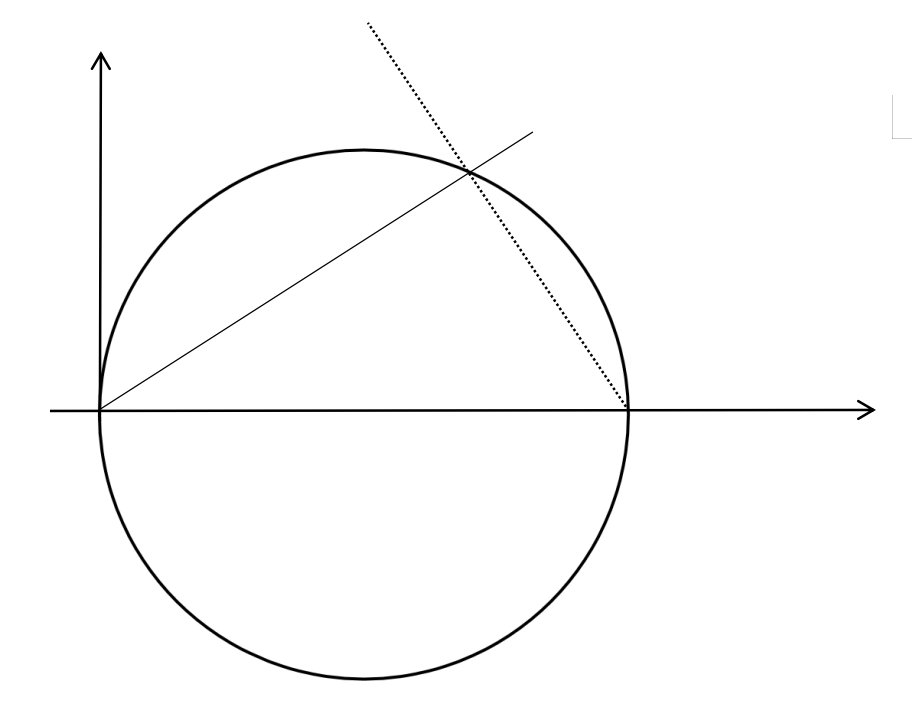}}
		\put(-43,180){\small$v$}
		\put(193,87){\small$u$}
		\put(129,71){\tiny$(\bar{q}, 0)$}
		\put(86,150){\tiny$(u, v)$}
		\put(103,165){\small$\mathsf{S}$}
		\put(-45,70){\small$O$}
		\end{picture}
	\end{center}
	\caption{shock polar for $\underline{u}=\bar{q}$}\label{fig:ShockPolar0}
\end{figure}

Particularly, we consider a limit case for $\mathbf{(P)}$. If the incoming flow has limit speed $\underline{u}=\bar{q}$, then, by \eqref{eq:BernoulliLaw}, the corresponding density is $\underline{\rho}=0$. By the entropy condition, the down stream density $\rho>0$ and the shock polar \eqref{eq:shockpolar} can be reduced to the circle
\begin{equation}\label{eq:ShockpolarS}
u^2+v^2-\bar{q}u=0.
\end{equation}
As a result, R-H condition \eqref{eq:RHCondtionb} becomes
\begin{equation*}
\phi'=\frac{u}{v}.
\end{equation*}
In this way, the direction of the shock front is tangent to the downstream velocity, and no mass pass through the shock wave. Thus we can construct a distribution solution by setting $\phi=b$ and the downstream domain degenerates into a curve $\mathsf{S}=\mathsf{W}$. The downstream velocity is defined along $\mathsf{W}$ by combining \eqref{eq:shockpolar} and \eqref{eq:WallCondition} for $v>0$. This limit solution satisfies the free boundary value problem $\mathbf{(P)}$ in weak sense and can be introduced for many obstacle with large curved boundary. Employing the continuity method, we attempt to construct a solution where $\mathsf{S}$ and $\mathsf{W}$ are separated, for $0 < \bar{q} - \underline{u} << 1$, as a small perturbation of this limit solution (see Figures \ref{fig:curvedwedges}-\ref{fig:bluntbody}). For the supersonic shock case, where $b'>\sqrt{\frac{\gamma-1}{2}}$, this problem has been discussed in \cite{Hu2018JMAA,HuQuJLMS2025,HuZhang2019SIAM}.
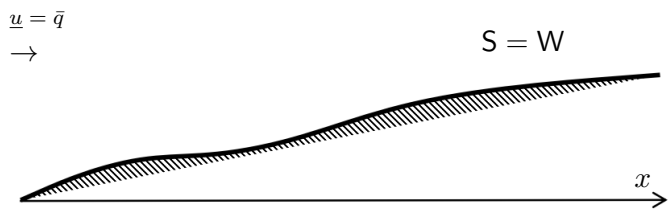
\begin{figure}[htbp]
	\setlength{\unitlength}{1bp}
	\begin{center}
		\begin{picture}(190,100)
		\put(-60,-10){\includegraphics[scale=0.4]{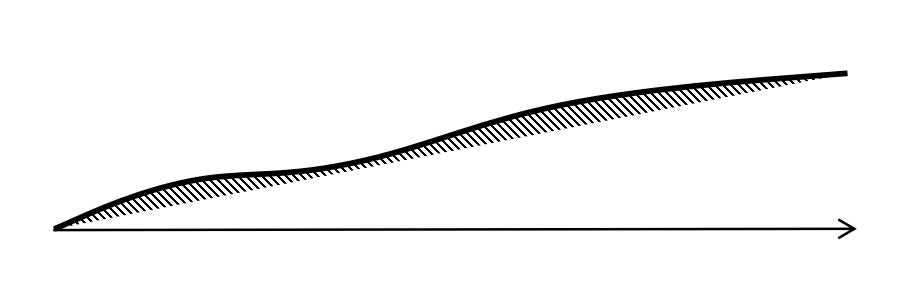}}
		\put(128,69){\small$\mathsf{S}=\mathsf{W}$}
		\put(-48,65){\small$\rightarrow$}
		\put(-48,79){\tiny$\underline{u}=\bar{q}$}
		\put(185,18){\small$x$}
		\end{picture}
	\end{center}
	\caption{limit solution for $\underline{u}=\bar{q}$}\label{fig:curvedwedges}
\end{figure}
\begin{figure}[htbp]
	\setlength{\unitlength}{1bp}
	\begin{center}
		\begin{picture}(190,100)
		\put(-60,-10){\includegraphics[scale=0.3]{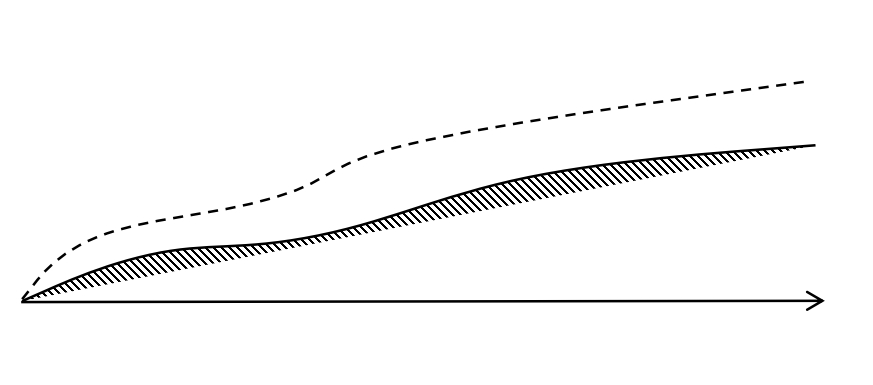}}
		\put(173,87){\small$\mathsf{S}$}
		\put(173,47){\small$\mathsf{W}$}
		\put(-48,65){\small$\rightarrow$}
		\put(-48,79){\tiny$\underline{u}<\bar{q}$}
		\put(175,20){\small$x$}
		\end{picture}
	\end{center}
	\caption{solution for $\underline{u}<\bar{q}$}\label{fig:curvedwedge}
\end{figure}
\begin{figure}[htbp]
	\setlength{\unitlength}{1bp}
	\begin{center}
		\begin{picture}(190,120)
		\put(30,-10){\includegraphics[scale=0.4]{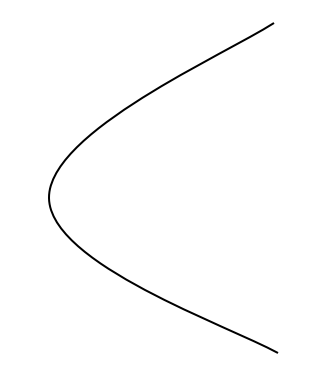}}
		\put(123,97){\small$\mathsf{S}=\mathsf{W}$}
		\put(-48,65){\small$\rightarrow$}
		\put(-48,79){\tiny$\underline{u}=\bar{q}$}
		\end{picture}
	\end{center}
	\caption{limit solution for $\underline{u}=\bar{q}$}\label{fig:bluntbodyS}
\end{figure}
\begin{figure}[htbp]
	\setlength{\unitlength}{1bp}
	\begin{center}
		\begin{picture}(190,190)
		\put(30,-20){\includegraphics[scale=0.4]{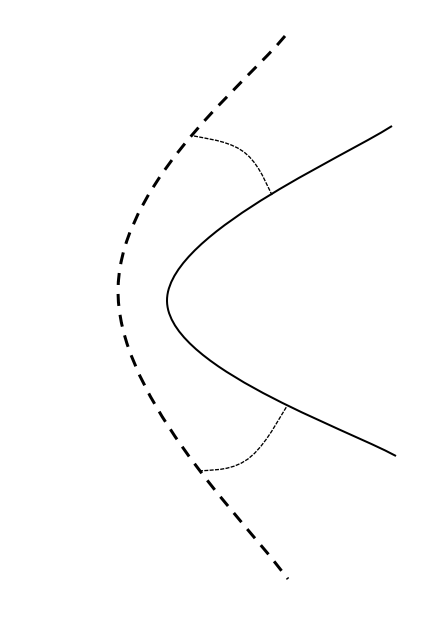}}
		\put(123,167){\small$\mathsf{S}$}
		\put(93,129){\tiny sonic curve}
		\put(155,132){\small$\mathsf{W}$}
		\put(-48,82){\small$\rightarrow$}
		\put(-48,96){\tiny$\underline{u}<\bar{q}$}
		\end{picture}
	\end{center}
	\caption{solution for $\underline{u}<\bar{q}$}\label{fig:bluntbody}
\end{figure}
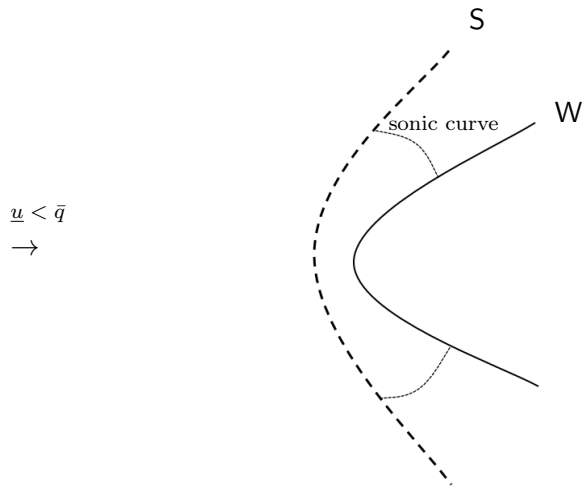

In addition, careful observation shows that the intersection points between \eqref{eq:ShockpolarS} and line $u=b_0v$ include not only $(u, v)$ but also $O=(0, 0)$. Thus we can define another limiting solution by setting $(u, v)=(0, 0)$ and $\mathsf{S}:{(x, y):~x=x_O}$, where $x_O<0$ is a constant. It can be proved that this solution satisfies $\mathbf{(P)}$ in distribution sense as well. Based on it, the authors construct a detached transonic shock solution in \cite{BaeXiang2023AAM} for $0<\bar{q}-\underline{u}<<1$.

The paper is organized as follows: In Section 2, by regarding the coordinate $y$ as the unknown function, we introduce the hodograph transformation, reformulate the three governing conditions \eqref{eq:EulerEquationsNC}, \eqref{eq:RHCondtionb} and \eqref{eq:WallCondition} within this framework, and state our main results in the $(u, v)$-coordinates. In Section 3, we presents preliminary calculations under fundamental assumptions and derives expressions of the limiting solution. In Section 4, we construct the approximate free boundary and corresponding approximate solution.

\section{hodograph transformation and main theorem}\label{sect:HodographTransformation}
The hodograph transformation is a fundamental method for solving partial differential equations, see \cite{CourantFriedrichs1976AMS,Evans2010GSM,Morawetz1964CPAM,Morawetz2004JHDE}. For example, it has been applied to the research on simple wave in \cite{CourantFriedrichs1976AMS}. There are many advantages to the hodograph transformation for our equation: the nonlinear equation is convert to a linear one, the degenerate curve of the equation is fixed and the shock free boundary value condition might be transformed into a fixed boundary value condition, and so on. However, it also brings about some difficulties. A typical case is that the solid wall condition is converted into a free boundary condition, which is difficult to handle. Here, we introduce a variant of the hodograph transformation, where we no longer consider the potential function or the stream function as unknown function. Instead, we directly consider a coordinate as the unknown function. In this way, the solid wall boundary condition is transformed into a free boundary form consists of a Dirichlet boundary condition and a slope boundary condition, which has been previously investigated, see \cite{ChenFeldman2003JAMS}. Furthermore, the hodograph transformation has corresponding conditions of use, in which the velocity needs to have a significant variation. Our problem of hypersonic flow onto a large curved wedge happens to satisfy this condition. We believe that the approach adopt here can also be applied in numerical analysis.
\subsection{Second order linear equation of $y$ under hodograph transformation}
First, we transform the elliptic system into a second order equation of the coordinate $y$. In fact, if the map $(x, y)\rightarrow(u, v)$ is $C^2$ and nondegenerate, then its Jacobian is
\begin{equation*}
\begin{pmatrix}
u_x&u_y\\
v_x&v_y
\end{pmatrix}
=
\begin{pmatrix}
x_u&x_v\\
y_u&y_v
\end{pmatrix}^{-1}
=
\begin{pmatrix}
y_v&-x_v\\
-y_u&x_u
\end{pmatrix}/
\begin{vmatrix}
x_u&x_v\\
y_u&y_v
\end{vmatrix}=
\begin{vmatrix}
u_x&u_y\\
v_x&v_y
\end{vmatrix}
\begin{pmatrix}
y_v&-x_v\\
-y_u&x_u
\end{pmatrix}.
\end{equation*}
Let $J:=\begin{vmatrix}
u_x&u_y\\
v_x&v_y
\end{vmatrix}$, we have
\begin{equation}\label{eq:HodographJacobi}
u_x=Jy_v,~u_y=-Jx_v,~v_x=-Jy_u,~v_y=Jx_u.
\end{equation}
Therefore, \eqref{eq:EulerEquationsM} can be reduced to
\begin{equation*}
	\begin{pmatrix}
		y_v\\
		-y_u
	\end{pmatrix}
	+ 
	\begin{pmatrix}
		\frac{-2uv}{c^2-u^2}& \frac{c^2-v^2}{c^2-u^2}\\
		-1& 0
	\end{pmatrix}
	\begin{pmatrix}
		-x_v\\
		x_u
	\end{pmatrix}=0.
\end{equation*}
provided $J$ is nonsingular. We can solve $x_u$ and $x_v$ in above as follows
\begin{equation}\label{eq:DuXExpression}
\begin{cases}
\displaystyle x_u=-\frac{2uv}{c^2-v^2}y_u-\frac{c^2-u^2}{c^2-v^2}y_v;\\
\displaystyle x_v=y_u.
\end{cases}
\end{equation}
In view of
\begin{equation*}
\qnt{x_u}_v=\qnt{x_v}_u,
\end{equation*}
we have
\begin{equation*}
\qnt{y_u}_u+\qnt{\frac{2uv}{c^2-v^2}y_u+\frac{c^2-u^2}{c^2-v^2}y_v}_v=0.
\end{equation*}
Finally, we get the {\bf second order linear equation of $y$} as follows
\begin{equation}\label{eq:ySecondOrderEquation}
\begin{cases}
\displaystyle (c^2-v^2)y_{uu}+2uvy_{uv}+(c^2-u^2)y_{vv}+C_1y_u+C_2y_v=0;\\
\displaystyle C_1=\frac{2\gamma v^2+2c^2}{c^2-v^2}u;\\
\displaystyle C_2=\frac{(2c^2-u^2-v^2)-\gamma(u^2-v^2)}{c^2-v^2}v.
\end{cases}
\end{equation}
This second-order linear equation of $y$ is also a mixed-type equation: it is an elliptic equation in the subsonic region and hyperbolic in the supersonic region. Unlike quasilinear equation \eqref{eq:EulerEquationsNC}, this equation of $y$ has fixed regions of equation type.

\subsection{The fixed boundary value condition of $y$ under hodograph transformation for the shock front}
For the incoming flow with the constant velocity $(\underline{u}, 0)$, we introduce the parameter $\epsilon$ of incoming mass flux defined via
\begin{equation*}
\underline{\rho}\underline{u}=\qnt{\frac{\qnt{\gamma-1}\qnt{\bar{q}^2-\underline{u}^2}}{2\gamma}}^{\frac{1}{\gamma-1}}\underline{u}=\epsilon.
\end{equation*}
Then for $\abs{\underline{u}-\bar{q}}<<1$, the states of the incoming flow can be regarded as functions of $\epsilon<<1$ and
\begin{equation*}
\underline{u}\rightarrow\bar{q}-\quad\mbox{as}\quad\epsilon\rightarrow0+.
\end{equation*}
We rewrite \eqref{eq:shockpolar} in the following form
\begin{equation}\label{eq:ShockPolarGExpressione}
G(u, v, \epsilon):=(u-\frac{\epsilon}{\rho})(u-\underline{u}(\epsilon))+v^2=0.
\end{equation}
Differentiate it along the shock front
\begin{equation*}
G_u(u_x\dif x+u_y\dif y)+G_v(v_x\dif x+v_y\dif y)=0.
\end{equation*}
By \eqref{eq:RHCondtionb} and the hodograph expressions \eqref{eq:HodographJacobi}, we can reduce above to
\begin{equation*}
-G_u([v]y_v+[u]x_v)+G_v([v]y_u+[u]x_u)=0,
\end{equation*}
or equivalently
\begin{equation*}
-G_u(-[\rho u]y_v+[\rho v]x_v)+G_v(-[\rho u]y_u+[\rho v]x_u)=0.
\end{equation*}
By \eqref{eq:DuXExpression}, we can replace $\nabla_{(u, v)}x$ terms by $\nabla_{(u, v)}y$ to get the oblique derivative condition of $y$ along the shock polar \eqref{eq:ShockPolarGExpressione} in $(u, v)$-plane
\begin{equation}\label{eq:yShockObliqueDerivativeU}
\begin{split}
&\qnt{G_u(c^2-v^2)[u]-G_v(c^2-v^2)[v]+G_v[u](2uv)}y_u\\
&\quad\quad+\qnt{G_u(c^2-v^2)[v]+G_v(c^2-u^2)[u]}y_v=0,
\end{split}
\end{equation}
or equivalently
\begin{equation}\label{eq:yShockObliqueDerivativeRU}
\begin{split}
\qnt{G_u(c^2-v^2)[\rho v]+G_v(c^2-v^2)[\rho u]+G_v[\rho v](2uv)}y_u\\
\quad\quad\quad\quad+\qnt{-G_u(c^2-v^2)[\rho u]+G_v(c^2-u^2)[\rho v]}y_v=0.
\end{split}
\end{equation}
\subsection{The free boundary value condition of $y$ under hodograph transformation for the solid wall}
Via the velocity fields $(x, y)\mapsto(u, v)$, on the $(u, v)$ plane, the solid wall boundary is mapped to the curve $\mathcal{W}:u=u_b(v)$ on $(u, v)$ plane. For the {\bf slip boundary condition} \eqref{eq:WallCondition}, by convexity, the function $b'(y)$ is a monotonically increasing function from $[0, +\infty)$ to $[\flat, \sharp)$. Then we have
\begin{equation}\label{eq:yWallVelocityCondition}
y=\qnt{b'}^{-1}(\frac{u_b}{v})=:B(\frac{u_b}{v}).
\end{equation}
By \eqref{eq:WallCondition2nd}, $B\in C^\infty$ is an increasing function and
\begin{equation}\label{eq:BRange}
B:[\flat, \sharp)\mapsto[0,+\infty).
\end{equation}
In this paper, we further assume that
\begin{equation}\label{eq:Bassumptions}
\begin{cases}
B''>0;\\
B(k)=\frac{1}{\qnt{\sharp-k}^\alpha},\quad\mbox{for}\quad\sharp-\delta<k<\sharp,
\end{cases}
\end{equation}
where $\delta$ is a positive constant and belongs to $(0, \sharp-\flat)$. From the assumptions of $\mathsf{W}$, we know that the wedge has a slope of $\frac{1}{\sharp}$ at infinity. However, calculations reveal that different values of $\alpha$ represent two distinct obstacle shapes: (1) When $0<\alpha<1$, the wedge has an asymptote line $L$ with slope $\frac{1}{\sharp}$ at infinity. (2) When $\alpha>1$, the slope of the wedge tends towards $\frac{1}{\sharp}$ at infinity but has no asymptote; in this case, the wedge cross every straight line with slope $\frac{1}{\sharp}$. In fact, on the wedge, we have
\begin{equation*}
k=\frac{\dif x}{\dif y}.
\end{equation*}
Then \eqref{eq:yWallVelocityCondition} with \eqref{eq:Bassumptions} leads to
\begin{equation*}
\frac{\dif x}{\dif y}=\sharp-y^{-\frac{1}{\alpha}}.
\end{equation*}
It follows
\begin{equation*}
x=\sharp y-\frac{\alpha}{\alpha-1}y^{1-\frac{1}{\alpha}}+C\quad\mbox{for}\quad x>>1.
\end{equation*}
This curve has an asymptote line for $0<\alpha<1$.

\begin{figure}[htbp]
	\setlength{\unitlength}{1bp}
	\begin{center}
		\begin{picture}(190,190)
		\put(-60,-10){\includegraphics[scale=0.4]{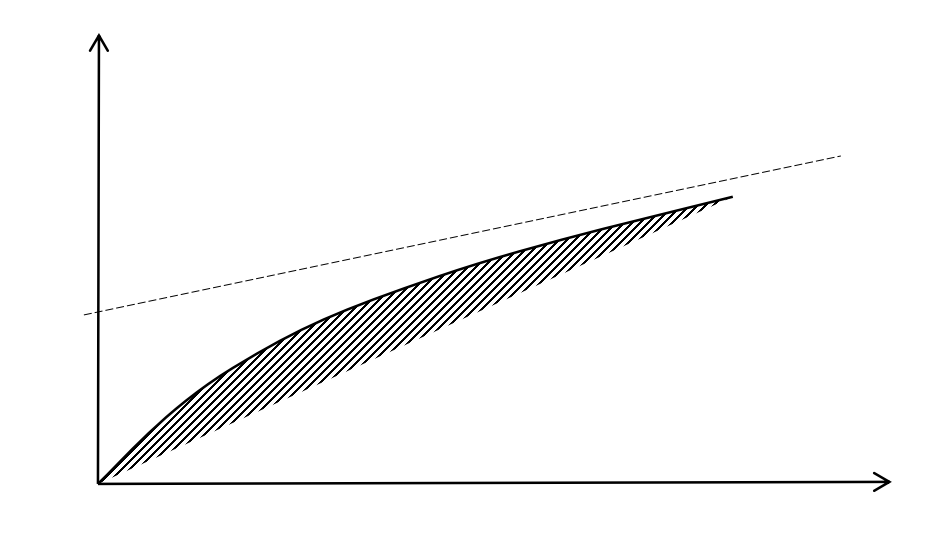}}
		\put(-42,145){\small$y$}
		\put(153,82){\small$\mathsf{W}$}
		\put(-23,72){\small$L$}
		\put(193,20){\small$x$}
		\put(-44,2){\small$O$}
		\end{picture}
	\end{center}
	\caption{$0<\alpha<1$}\label{fig:alphal1}
\end{figure}
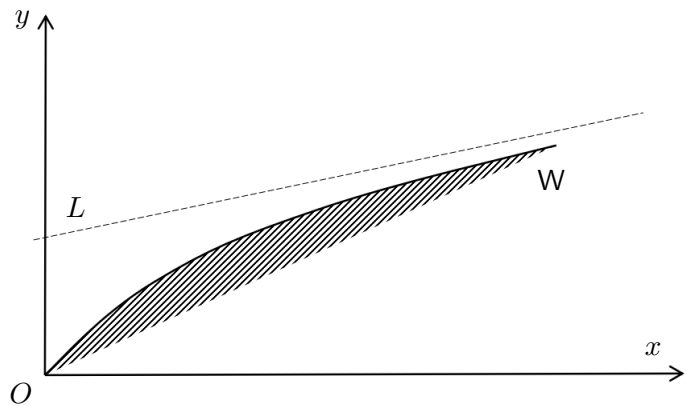
\begin{figure}[htbp]
	\setlength{\unitlength}{1bp}
	\begin{center}
		\begin{picture}(190,190)
		\put(-60,-10){\includegraphics[scale=0.4]{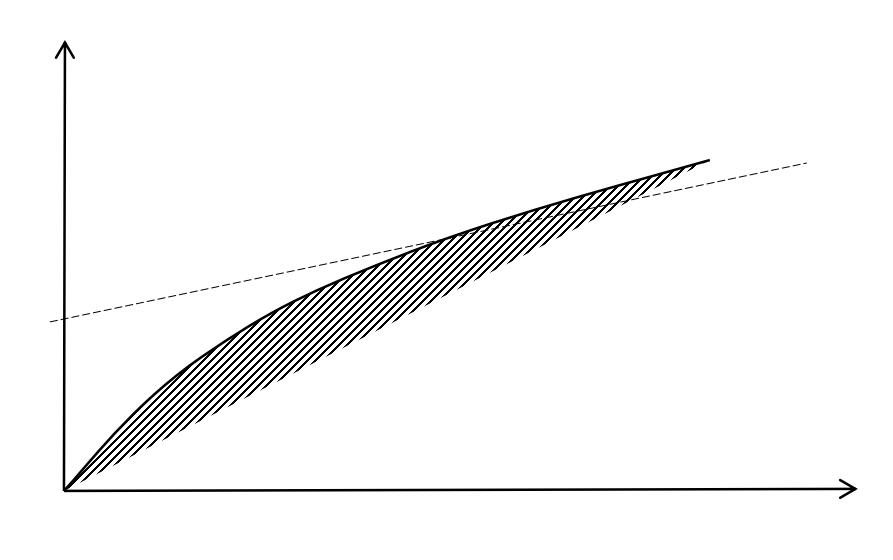}}
		\put(-53,140){\small$y$}
		\put(123,108){\small$\mathsf{W}$}
		\put(-23,72){\small$L$}
		\put(190,20){\small$x$}
		\put(-52,0){\small$O$}
		\end{picture}
	\end{center}
	\caption{$\alpha>1$}\label{fig:alphag1}
\end{figure}

For the {\bf wall position condition} \eqref{eq:WallPositionCondition}, under the hodograph transformation, recalling $(x, y)$ is a vector function of $(u, v)$, for \eqref{eq:WallPositionCondition}, we differentiate it with respect to $v$ on $(u, v)$ plane to have
\begin{equation*}
x_v+x_u\frac{\dif u_b}{\dif v}=\frac{u}{v}(y_v+y_u\frac{\dif u_b}{\dif v}),
\end{equation*}
where we apply \eqref{eq:WallCondition} to replace $b'(y)$ by $\frac{u}{v}$. Then substitute \eqref{eq:DuXExpression} into above, we get the free boundary value condition corresponding to the solid wall
\begin{equation}\label{eq:yWallPositionCondition}
-\qnt{v\qnt{\qnt{2uv}y_u+\qnt{c^2-u^2}y_v}+\qnt{c^2-v^2}uy_u}\frac{\dif u_b}{\dif v}=\qnt{c^2-v^2}\qnt{uy_v-vy_u},
\end{equation}
where $u=u_b(v)$.

\subsection{Main theorem}
To conclude, we reduce the free boundary value problem \eqref{eq:EulerEquationsNC} with \eqref{eq:RHCondtionb} and \eqref{eq:WallCondition} to the free boundary value problem for the unknown function $y$ on $(u, v)$ plane as follows:

\begin{itemize}
\item 
For $y$, we get the second order differential equation \eqref{eq:ySecondOrderEquation}.

\item 
For the shock, we get the oblique derivative boundary value condition \eqref{eq:yShockObliqueDerivativeU} ( or equivalent \eqref{eq:yShockObliqueDerivativeRU}) along the shock polar \eqref{eq:ShockPolarGExpressione}.

\item 
For the wall, we get the free boundary $\mathcal{W}$ with
{\bf wall position condition} \eqref{eq:yWallPositionCondition} and the {\bf slip boundary condition} \eqref{eq:yWallVelocityCondition}.
\end{itemize}
In the following, we abbreviate this problem as $\mathbf{(P_y)}$. Then, our main result is

\begin{figure}[htbp]
\setlength{\unitlength}{1bp}
\begin{center}
\begin{picture}(190,190)
\put(-60,-10){\includegraphics[scale=0.4]{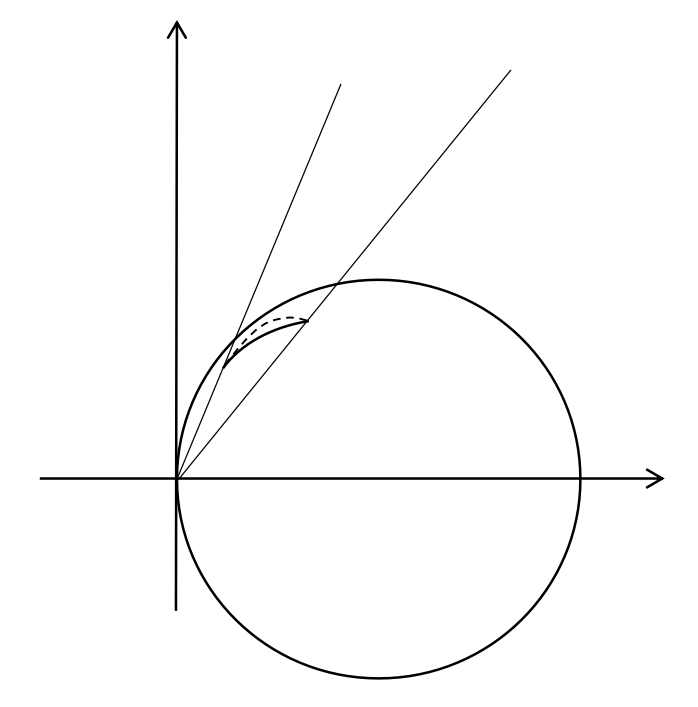}}
\put(-19,183){\small$v$}
\put(25,180){\tiny$\frac{u}{v}=\flat$}
\put(78,184){\tiny$\frac{u}{v}=\sharp$}
\put(-19,183){\small$v$}
\put(93,52){\tiny$u=\bar{q}$}
\put(90,110){\tiny$\mathcal{S}:\epsilon=0$}
\put(31,101){\tiny$\mathcal{S}:\epsilon<<1$}
\put(133,65){\small$u$}
\put(-18,50){\small$O$}
\end{picture}
\end{center}
\caption{dashed line: $\mathcal{W}$}\label{fig:shockpolarP}
\end{figure}
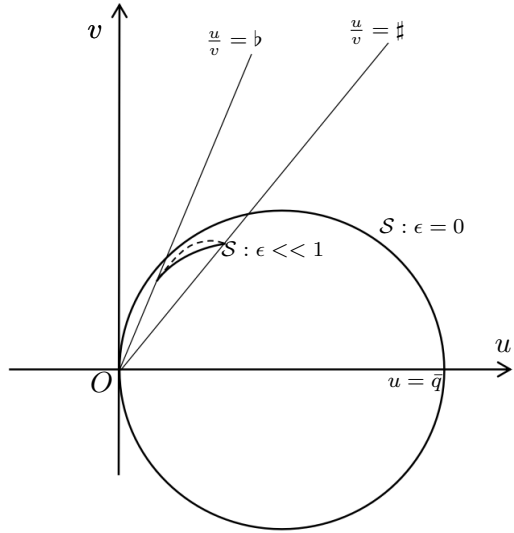
\begin{thm}\label{thm:MainTheorem}
For the problem $\mathbf{(P_y)}$ with $\gamma=2$, we can construct an approximate solution consist of a wall boundary $\mathcal{W}^\Gamma$ and a solution $y^\Gamma$ provided $\epsilon<\epsilon_0$, where $\epsilon_0$ is depending on the wall. $y^\Gamma$ satisfies \eqref{eq:ySecondOrderEquation} with the boundary value condition \eqref{eq:yShockObliqueDerivativeU} on \eqref{eq:ShockPolarGExpressione} and the Dirichlet condition \eqref{eq:yWallVelocityCondition} on $\mathcal{W}^\Gamma$. Along the wall boundary $\mathcal{W}^\Gamma$, for \eqref{eq:yWallPositionCondition}, we have the following approximate condition
\begin{equation}\label{eq:ApproximateBoundaryEstimate}
\abs{\frac{\dif u_b}{\dif v}+\qnt{c^2-v^2}\qnt{uy_v-vy_u}/\qnt{v\qnt{\qnt{2uv}y_u+\qnt{c^2-u^2}y_v}+\qnt{c^2-v^2}uy_u}}\leq C\epsilon^{1+\beta},
\end{equation}
where $\beta\in(0, 1)$ is a constant. Meanwhile, we can define $x^\Gamma$ via \eqref{eq:DuXExpression} and the map from $(u, v)$ to $(x^\Gamma, y^\Gamma)$ is invertible.
\end{thm}
\begin{rem}
We do not obtain an exact solution, but only an approximate one. This solution has an error that is a higher-order infinitesimal than the flux. Moreover, depending on the narrow ellipse estimation, the error constant in \eqref{eq:ApproximateBoundaryEstimate} is uniform. Due to the extreme difficulty in calculating highly disturbed flow fields, our findings offer valuable insights applicable to solving more general problems.
\end{rem}
\section{Some settings and limit solution}
Flow field with significant deflection is tough to calculate. We often need to compute bounds of some quantities across a vast range, which is difficult via Taylor series expansion. Here, we perform the calculation under specific simplifying assumptions. One is $\gamma=2$, where both density and the square of the speed of sound become polynomials in velocity. Consequently, many quantities reduce to polynomials, whose elimination relations are relatively simple and amenable to symbolic computation software, see \cite{Hu2025,HuQuJLMS2025}. Another is the hypersonic approximation, where the shock polar simplifies to a circle, leading to straightforward relationships among variables that can be computed exactly, see \cite{Hu2018JMAA,HuZhang2019SIAM}.
\subsection{$\gamma=2$ and $\epsilon=0$}
Under the assumption $\gamma=2$, pressure $p=\rho^2$ and sonic $c^2=\frac{\dif p}{\dif \rho}=2\rho$, the Bernoulli's law \eqref{eq:BernoulliLaw} becomes
\begin{equation}\label{eq:BernoulliLaw2}
\frac{u^2+v^2}{2}+2\rho=\frac{\bar{q}^2}{2}.
\end{equation}
Thus
\begin{equation}\label{eq:RhoCExpression2}
\rho=\frac{\bar{q}^2-q^2}{4},\quad c^2=\frac{\bar{q}^2-q^2}{2}.
\end{equation}
For the sonic point, we have
\begin{equation*}
q^2=c^2=\frac{\bar{q}^2-q^2}{2}.
\end{equation*}
Then the critical sonic is $c_0=\frac{\sqrt{3}}{3}\bar{q}$. It can be proved that the equation is elliptic for $q<c_0$ and hyperbolic for $q>c_0$.

We have
\begin{equation}\label{eq:RhoCDifExpression2}
\qnt{\rho_u, \rho_v}=-\frac{1}{2}\qnt{u, v},\quad \qnt{\qnt{c^2}_u, \qnt{c^2}_v}=-\qnt{u, v}.
\end{equation}
The equation \eqref{eq:ySecondOrderEquation} becomes
\begin{equation}\label{eq:ySecondOrderEquation2}
\begin{cases}
\displaystyle(c^2-v^2)y_{uu}+2uvy_{uv}+(c^2-u^2)y_{vv}+C_1y_u+C_2y_v=0;\\
\displaystyle C_1=\frac{4 v^2+2c^2}{c^2-v^2}u;\\
\displaystyle C_2=\frac{(2c^2-u^2-v^2) - 2(u^2-v^2)}{c^2-v^2}v.
\end{cases}
\end{equation}
For $\epsilon=0$, the shock polar \eqref{eq:ShockPolarGExpressione} is reduced to \eqref{eq:ShockpolarS}.
The upper sonic point on it is 
\begin{equation*}
(u, v)=\qnt{\frac{\bar{q}}{3}, \frac{\sqrt{2}\bar{q}}{3}}.
\end{equation*}

\subsection{Limit solution $\Upsilon$}
In this part, we introduce the limit solution $\Upsilon$ for $\mathbf{(P_y)}$. It is constructed by restricting all conditions to \eqref{eq:ShockpolarS}. Starting from this section, we carry out many calculations using the computational software Mathematica. In the programme, the variables $u, v, \underline{u}, \epsilon, \bar{q}$ are recorded as
\begin{lstlisting}
u, v, U, e, Q
\end{lstlisting}
respectively. Then we record $c^2$ and $\rho$ as follows functions
\begin{lstlisting}
CC[u, v], Rho[u, v]
\end{lstlisting}
respectively. Our setting is
\begin{lstlisting}
Clear[u, v, U, e];
Rho[u_, v_] := (Q^2 - u^2 - v^2)/4;
CC[u_, v_] := (Q^2 - u^2 - v^2)/2;
G[u_, v_, U_, e_] := (u - e/Rho[u, v]) (u - U) + v^2;
InF[U_, e_] := Rho[U, 0] U - e;
D[G[u, v, U, e], u];
Gu[u_, v_, U_, e_] := 
u - (4 e)/(
Q^2 - u^2 - 
v^2) + (u - U) (1 - (8 e u)/(Q^2 - u^2 - v^2)^2);
D[G[u, v, U, e], v];
Gv[u_, v_, U_, e_] := 
2 v - (8 e (u - U) v)/(Q^2 - u^2 - v^2)^2;
vv0u = {v^x_ /; x >= 2 -> u (Q - u) v^(x - 2)};
Ue0Q = {U -> Q, e -> 0};
RestrictS[expr_] := 
Module[{temp}, temp = expr /. Ue0Q; 
temp = Simplify[temp]; 
temp = temp //. vv0u; 
FullSimplify[temp]
];
RHu[u_, v_, U_, e_] := u - U;
RHv[u_, v_, U_, e_] := v;
\end{lstlisting}

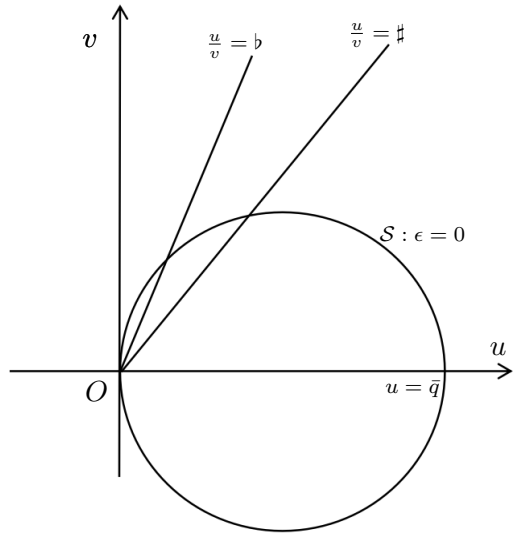
\begin{figure}[htbp]
\setlength{\unitlength}{1bp}
\begin{center}
\begin{picture}(190,190)
\put(-60,-10){\includegraphics[scale=0.4]{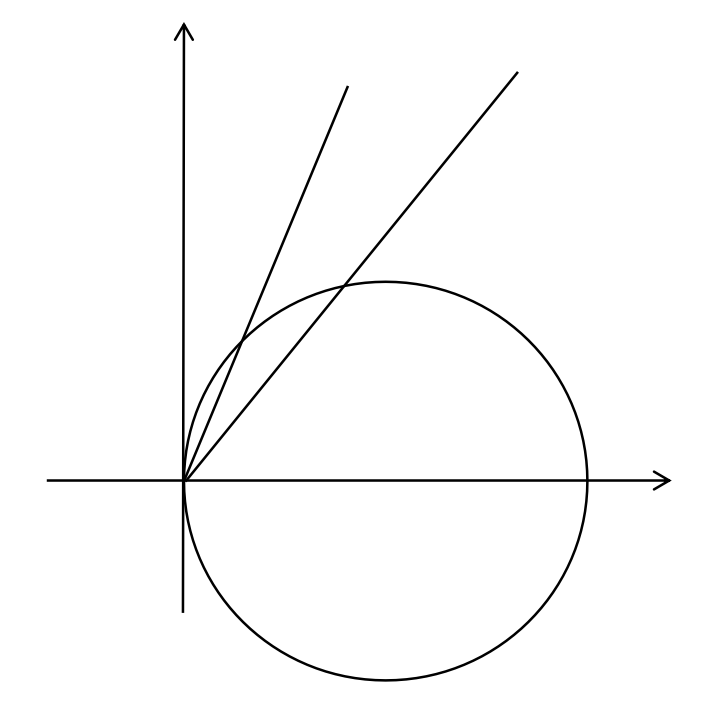}}
\put(-19,183){\small$v$}
\put(27,182){\tiny$\frac{u}{v}=\flat$}
\put(80,186){\tiny$\frac{u}{v}=\sharp$}
\put(-19,183){\small$v$}
\put(94,53){\tiny$u=\bar{q}$}
\put(92,110){\tiny$\mathcal{S}:\epsilon=0$}
\put(133,67){\small$u$}
\put(-18,50){\small$O$}
\end{picture}
\end{center}
\caption{Limit solution $\Upsilon$}\label{fig:shockpolarS}
\end{figure}

First, on the shock polar \eqref{eq:ShockpolarS}, 
\begin{equation*}
(G_u, G_v)=(2u-\bar{q}, 2v).
\end{equation*}
The coefficients of the oblique derivative condition \eqref{eq:yShockObliqueDerivativeU} are
\begin{equation*}
G_u(c^2-v^2)[u]-G_v(c^2-v^2)[v]+G_v[u](2uv)=-\bar{q}(u-\bar{q})^2\frac{(6u-\bar{q})}{2},
\end{equation*}
\begin{lstlisting}
RestrictS[
Gu[u, v, U, e] (CC[u, v] - v^2) RHu[u, v, U, e] - 
Gv[u, v, U, e] (CC[u, v] - v^2) RHv[u, v, U, e] + 
Gv[u, v, U, e] (2 u v) RHu[u, v, U, e]]
\end{lstlisting}
and
\begin{equation*}
G_u(c^2-v^2)[v]+G_v(c^2-u^2)[u]=-3\bar{q}\frac{(u-\bar{q})(2u-\bar{q})}{2}v.
\end{equation*}
\begin{lstlisting}
RestrictS[
Gu[u, v, U, e] (CC[u, v] - v^2) RHv[u, v, U, e] + 
Gv[u, v, U, e] (CC[u, v] - u^2) RHu[u, v, U, e]]
\end{lstlisting}
Thus we can deduce the oblique derivative condition of $\Upsilon$ on the shock polar \eqref{eq:ShockpolarS} as follows
\begin{equation}\label{eq:ShockPolarObliqueConditionS}
(u-\bar{q})(6u-\bar{q})\Upsilon_u+3(2u-\bar{q})v\Upsilon_v=0.
\end{equation}
Here and in subsequent discussions, particularly when working along \eqref{eq:ShockpolarS}, we apply the substitution rule: $v^2 \to u(\bar{q}-u)$.

Second, letting the wall free boundary $\mathcal{W}$ coincides to $\mathcal{S}$ and setting $\epsilon=0$, for the Dirichlet wall boundary value condition
\begin{equation}\label{eq:WallConditionBS}
\Upsilon=B(\frac{u}{v}),
\end{equation}
differentiating above along the shock polar \eqref{eq:ShockpolarS}, we have
\begin{equation*}
\Upsilon_u\frac{\dif u}{\dif v}+\Upsilon_v=B'(\frac{u}{v})(\frac{1}{v}\frac{\dif u}{\dif v}-\frac{u}{v^2}),
\end{equation*}
where
\begin{equation*}
\frac{\dif u}{\dif v}=-\frac{G_v}{G_u}=-\frac{2v}{2u-\bar{q}}.
\end{equation*}
We get the derivative form of \eqref{eq:WallConditionBS}
\begin{equation}\label{eq:WallConditionBS1}
2v\Upsilon_u-(2u-\bar{q})\Upsilon_v=B'(\frac{u}{v})\frac{\bar{q}}{\bar{q}-u}.
\end{equation}

\medskip
{\bf First order case.} Combine \eqref{eq:WallConditionBS1} with \eqref{eq:ShockPolarObliqueConditionS}, we have
\begin{equation*}
\begin{cases}
(u-\bar{q})(6u-\bar{q})\Upsilon_u+3(2u-\bar{q})v\Upsilon_v=0;\\
2v\Upsilon_u+(\bar{q}-2u)\Upsilon_v=B'(\frac{u}{v})\frac{\bar{q}}{\bar{q}-u}.
\end{cases}
\end{equation*}
This can be solved out
\begin{equation*}
\qnt{\Upsilon_u, \Upsilon_v}=B'(\frac{u}{v})\qnt{\frac{3v}{(\bar{q}-u)^2}, \frac{\bar{q}-6u}{(\bar{q}-2u)(\bar{q}-u)}}.
\end{equation*}
\begin{lstlisting}
{y0u, y0v}=RestrictS[
LinearSolve[{{(u - Q) (6 u - Q), 3 (2 u - Q) v}, {2 v, Q - 2 u}}, {0,
b  Q/(Q - u)}]]
\end{lstlisting}
where $\qnt{\Upsilon_u, \Upsilon_v}$ and $B'$ are recorded as
\begin{lstlisting}
{y0u, y0v}
\end{lstlisting}
and
\begin{lstlisting}
b
\end{lstlisting}
respectively.

Particular, for the wall position condition \eqref{eq:yWallPositionCondition}, restricting the coefficients on the shock polar \eqref{eq:ShockpolarS}, we have
\begin{equation}\label{eq:WallBoundaryPositionS}
-\qnt{v\frac{(u+\bar{q})(\bar{q}-2u)}{2}y_v+\frac{\qnt{\bar{q}-u}\qnt{\bar{q}+2u}}{2}uy_u}\frac{\dif u}{\dif v}=\frac{(u-\bar{q})(2u-\bar{q})}{2}\qnt{uy_v-vy_u}.
\end{equation}
This formula is naturally satisfied for $\qnt{y_u, y_v}=\qnt{\Upsilon_u, \Upsilon_v}$ along \eqref{eq:ShockpolarS}.

\medskip
{\bf Second order case.} First, for \eqref{eq:ShockPolarObliqueConditionS}, differentiate it along \eqref{eq:ShockpolarS} with respect to $v$, we get
\begin{equation*}
\frac{\dif}{\dif v}\qnt{(u-\bar{q})(6u-\bar{q})\Upsilon_u+3(2u-\bar{q})v\Upsilon_v}=0.
\end{equation*}
It is
\begin{equation*}
\begin{split}
&(u-\bar{q})(6u-\bar{q})\qnt{-\frac{2v}{2u-\bar{q}}\Upsilon_{uu}+\Upsilon_{uv}}+3(2u-\bar{q})v\qnt{-\frac{2v}{2u-\bar{q}}\Upsilon_{uv}+\Upsilon_{vv}}\\
&\quad+\qnt{-\frac{2v}{2u-\bar{q}}(6u-\bar{q})-\frac{12v}{2u-\bar{q}}(u-\bar{q})}\Upsilon_u+\qnt{-\frac{12v^2}{2u-\bar{q}}+6u-3\bar{q}}\Upsilon_v=0.
\end{split}
\end{equation*}
This can be simplified as follows
\begin{equation}\label{eq:ShockPolarObliqueConditionS2}
\begin{split}
&2v(\bar{q}-u)(\bar{q}-6u)\Upsilon_{uu}+(\bar{q}-u)(\bar{q}-12u)(\bar{q}-2u)\Upsilon_{uv}-3(\bar{q}-2u)^2v\Upsilon_{vv}\\
&\quad+2v(12u-7\bar{q})\Upsilon_u-3\qnt{8u^2-8u\bar{q}+\bar{q}^2}\Upsilon_v=0.
\end{split}
\end{equation}

Second, differentiate \eqref{eq:WallConditionBS1} along \eqref{eq:ShockpolarS} with respect to $v$, we have
\begin{equation*}
\frac{\dif }{\dif v}\qnt{2v\Upsilon_u+(\bar{q}-2u)\Upsilon_v}=\frac{\dif }{\dif v}\qnt{B'(\frac{u}{v})\frac{\bar{q}}{\bar{q}-u}}.
\end{equation*}
It is
\begin{equation*}
\begin{split}
&2v\qnt{-\frac{2v}{2u-\bar{q}}\Upsilon_{uu}+\Upsilon_{uv}}+(\bar{q}-2u)\qnt{-\frac{2v}{2u-\bar{q}}\Upsilon_{uv}+\Upsilon_{vv}}+2\Upsilon_u+\frac{4v}{2u-\bar{q}}\Upsilon_v\\
&=B''(\frac{u}{v})\frac{u'v-u}{v^2}\frac{\bar{q}}{\bar{q}-u}+B'(\frac{u}{v})\frac{\bar{q}}{(\bar{q}-u)^2}u'\\
&=B''(\frac{u}{v})\frac{\bar{q}}{u(\bar{q}-u)^2}\qnt{-\frac{2v^2}{2u-\bar{q}}-u}-B'(\frac{u}{v})\frac{\bar{q}}{(\bar{q}-u)^2}\frac{2v}{2u-\bar{q}}\\
&=B''(\frac{u}{v})\frac{\bar{q}}{u(\bar{q}-u)^2}\qnt{-\frac{2v^2+2u^2-u\bar{q}}{2u-\bar{q}}}-B'(\frac{u}{v})\frac{\bar{q}}{(\bar{q}-u)^2}\frac{2v}{2u-\bar{q}}\\
&=-B''(\frac{u}{v})\frac{\bar{q}^2}{(\bar{q}-u)^2\qnt{2u-\bar{q}}}-B'(\frac{u}{v})\frac{2\bar{q}v}{(\bar{q}-u)^2\qnt{2u-\bar{q}}}.
\end{split}
\end{equation*}
This can be simplified as
\begin{equation}\label{eq:WallConditionBS2}
\begin{split}
&4v^2\Upsilon_{uu}-4v\qnt{2u-\bar{q}}\Upsilon_{uv}+\qnt{2u-\bar{q}}^2\Upsilon_{vv}-2\qnt{2u-\bar{q}}\Upsilon_u-4v\Upsilon_v\\
&=B''(\frac{u}{v})\frac{\bar{q}^2}{(\bar{q}-u)^2}+B'(\frac{u}{v})\frac{2\bar{q}v}{(\bar{q}-u)^2}.
\end{split}
\end{equation}

Finally, we restrict \eqref{eq:ySecondOrderEquation2} only on the shock polar \eqref{eq:ShockpolarS} to get the equation for $\Upsilon$ on it
\begin{equation}\label{eq:y2S}
\begin{cases}
\displaystyle\frac{(u-\bar{q})(2u-\bar{q})}{2}\Upsilon_{uu}+2uv\Upsilon_{uv}+\frac{(u+\bar{q})(\bar{q}-2u)}{2}\Upsilon_{vv}+C_1\Upsilon_u+C_2\Upsilon_v=0;\\
\displaystyle C_1=\frac{2\gamma v^2+2c^2}{c^2-v^2}u=\frac{4 (\bar{q}-u)u+\bar{q}^2-\bar{q}u}{\frac{\bar{q}^2-\bar{q}u}{2}+u(u-\bar{q})}u= \frac{2u\qnt{4u+\bar{q}}}{\bar{q}-2u};\\
\displaystyle C_2=\frac{(2c^2-u^2-v^2)-\gamma(u^2-v^2)}{c^2-v^2}v=2\frac{(\bar{q}^2-\bar{q}u-\bar{q}u)-2(2u^2-\bar{q}u)}{(u-\bar{q})(2u-\bar{q})}v\\
\displaystyle\quad=\frac{2(\bar{q}+2u)v}{\bar{q}-u}.
\end{cases}
\end{equation}
\begin{lstlisting}
RestrictS[(4 v^2 + 2 CC[u, v]) u/(CC[u, v] - v^2)]
RestrictS[((2 CC[u, v] - u^2 - v^2) - 
2 (u^2 - v^2)) v/(CC[u, v] - v^2)]
\end{lstlisting}
By \eqref{eq:RhoCExpression2}, along the shock polar \eqref{eq:ShockpolarS}, above equation is elliptic for
\begin{equation*}
u<\frac{\bar{q}}{3}.
\end{equation*}
Meanwhile, the slope
\begin{equation*}
\frac{u}{v}\leq\frac{\sqrt{2}}{2}.
\end{equation*}

Therefore, solving the linear equations \eqref{eq:ShockPolarObliqueConditionS2}, \eqref{eq:WallConditionBS2} and \eqref{eq:y2S}, we get the second order derivatives of $\Upsilon$ as follows
\begin{equation*}
\begin{cases}
\displaystyle\Upsilon_{uu}=-\frac{(\bar{q}-11u)B''}{(\bar{q}-u)^3}+\frac{4v(2\bar{q}-3u)(\bar{q}+4u) B'}{(\bar{q}-u)^3(\bar{q}-2u)\bar{q}};\\
\displaystyle\Upsilon_{uv}=\frac{(5\bar{q}-22u)vB''}{(\bar{q}-u)^3(\bar{q}-2u)}+\frac{\qnt{3\bar{q}^3-16u\bar{q}^2-52u^2\bar{q}+96u^3}B'}{\bar{q}(\bar{q}-u)^2(\bar{q}-2u)^2};\\
\displaystyle\Upsilon_{vv}=\frac{\qnt{\bar{q}^2-16u\bar{q}+44u^2}B''}{(\bar{q}-u)^2(\bar{q}-2u)^2}-\frac{4v\qnt{3\bar{q}^3-6u\bar{q}^2-32u^2\bar{q}+48u^3}B'}{\bar{q}(\bar{q}-u)^2(\bar{q}-2u)^3}.
\end{cases}
\end{equation*}
\begin{lstlisting}
{y0uu, y0uv, y0vv}=RestrictS[
LinearSolve[{{2 v (Q - u) (Q - 6 u), (Q - u) (Q - 12 u) (Q - 
2 u), -3 (Q - 2 u)^2 v}, {4 v^2, 
4 v (Q - 2 u), (Q - 2 u)^2}, {(u - Q) (2 u - Q)/2, 
2 u  v, (u + Q) (Q - 2 u)/2}}, {-2 v (12 u - 7 Q) y0u + 
3 (8 u^2 - 8 u  Q + Q^2) y0v, 
2 (2 u - Q) y0u + 4 v  y0v + c  Q^2/(Q - u)^2 + 
b  2 Q  v/(Q - u)^2, -2 u (4 u + Q)/(Q - 2 u) y0u - 
2 (Q + 2 u) v/(Q - u) y0v}]]
\end{lstlisting}
where $\qnt{\Upsilon_{uu}, \Upsilon_{uv}, \Upsilon_{vv}}$ and $B''$ are recorded as
\begin{lstlisting}
{y0uu, y0uv, y0vv}
\end{lstlisting}
and
\begin{lstlisting}
c
\end{lstlisting}
respectively.

\section{Approximate Solution}
In this section, we construct our approximate solution for $\mathbf{(P_y)}$ and prove Theorem \ref{thm:MainTheorem}. The proof hinges on a key observation: the slope-type free boundary condition \eqref{eq:yWallPositionCondition} depends on a quantity $\frac{y_u}{y_v}(u, v)$ (or the image of $\nabla y(u, v)$ in the project space $\qnt{\nabla y}^\circ(u, v)$), which is itself determined by condition \eqref{eq:ShockPolarObliqueConditionS} on the shock front $\mathcal{S}$. Indeed, close inspection of equation \eqref{eq:WallBoundaryPositionS} reveals that, for equation \eqref{eq:WallBoundaryPositionS} to hold, only condition \eqref{eq:ShockPolarObliqueConditionS} is required – solving \eqref{eq:ShockPolarObliqueConditionS} and \eqref{eq:WallConditionBS1} simultaneously to obtain $\nabla \Upsilon$ is unnecessary. Consequently, if the limit solution $\Upsilon$ represents the asymptotic state of solutions to problem $\mathbf{(P_y)}$, the error in its derivative is $\varepsilon(\approx\epsilon^\beta)$. Utilizing $\Upsilon$'s derivatives, we extend $\frac{y_u}{y_v}(u, v)$, which is originally defined on $\mathcal{S}$, to a neighbourhood of $\mathcal{S}$. Through expression \eqref{eq:yWallPositionCondition}, this yields an approximate slope vector field. Notably, the error between this approximate slope vector field and that of the exact solution scales as $\epsilon\varepsilon$ (where $\epsilon$ is the domain width), representing a higher-order infinitesimal compared to $\epsilon$. This approximate vector field then defines an approximate free boundary. Solving the associated boundary value problem \eqref{eq:ySecondOrderEquation}, \eqref{eq:yShockObliqueDerivativeU} and \eqref{eq:yWallVelocityCondition} yields the approximate solution. Finally, narrow-region elliptic estimates provide error bounds for this approximate solution.

\subsection{The quantity $\frac{y_u}{y_v}$ (or $\frac{y_v}{y_u}$) and the slope expressions}
In this part, we derive explicit expressions for quantities associated with $\qnt{\nabla y}^\circ(u, v)$. In the following, if two vectors $\vec{a}$ and $\vec{b}$ are parallel and share the same direction, i.e.
\begin{equation*}
\vec{a}//\vec{b},\quad\mbox{and}\quad\vec{a}\cdot\vec{b}>0,
\end{equation*}
then we introduce the notation $'///'$ as follows
\begin{equation*}
\vec{a}///\vec{b}.
\end{equation*}
For the solution $y(u, v)$ approaches $\Upsilon$ in second order, by the slope relation \eqref{eq:yWallPositionCondition}, we have the slope expression
\begin{equation}\label{eq:UpdateBoundaryCondition}
\begin{pmatrix}
y_u&y_v
\end{pmatrix}
\begin{pmatrix}
(c^2+v^2)u&-(c^2-v^2)v\\
(c^2-u^2)v&(c^2-v^2)u
\end{pmatrix}
\begin{pmatrix}
0&1\\
-1&0
\end{pmatrix}
///\begin{pmatrix}
\dif u&\dif v
\end{pmatrix}
///\begin{pmatrix}
\frac{\dif u}{\dif v}&1
\end{pmatrix}.
\end{equation}
Define a vector $s'$ as follows
\begin{equation}\label{eq:s}
\begin{split}
\begin{pmatrix}
\frac{\dif u}{\dif v}&1
\end{pmatrix}
///s' :&=
\begin{pmatrix}
y_u&y_v
\end{pmatrix}
\begin{pmatrix}
(c^2+v^2)u&-(c^2-v^2)v\\
(c^2-u^2)v&(c^2-v^2)u
\end{pmatrix}
\begin{pmatrix}
0&1\\
-1&0
\end{pmatrix}\\
&=\begin{pmatrix}
y_u&y_v
\end{pmatrix}
\begin{pmatrix}
(c^2-v^2)v&(c^2+v^2)u\\
-(c^2-v^2)u&(c^2-u^2)v
\end{pmatrix}.
\end{split}
\end{equation}
Consider inner product of the outer normal vector $\nabla G=(G_u, G_v)$($\approx\qnt{2u-\bar{q}, 2v}$) and $s'$, we have
\begin{equation}\label{eq:s'DG}
\begin{split}
s' \cdot \nabla_U G&=\begin{pmatrix}
y_u&y_v
\end{pmatrix}
\begin{pmatrix}
(c^2-v^2)v&(c^2+v^2)u\\
-(c^2-v^2)u&(c^2-u^2)v
\end{pmatrix}
\begin{pmatrix}
G_u\\G_v
\end{pmatrix}\\
&=\qnt{(c^2-v^2)vG_u+(c^2+v^2)uG_v}y_u+\qnt{(c^2-u^2)vG_v-(c^2-v^2)uG_u}y_v.
\end{split}
\end{equation}

Furthermore, on the shock boundary $\mathcal{S}$, we have the oblique derivative condition \eqref{eq:yShockObliqueDerivativeU}. It can be written in the matrix form as follows
\begin{equation}\label{eq:yShockObliqueDerivativeUM}
\begin{pmatrix}
G_u&G_v
\end{pmatrix}
\begin{pmatrix}
(c^2-v^2)[u]&(c^2-v^2)[v]\\
(2uv)[u]-(c^2-v^2)[v]&(c^2-u^2)[u]
\end{pmatrix}
\begin{pmatrix}
y_u\\
y_v
\end{pmatrix}=0.
\end{equation}
Therefore, with above constrain, we have
\begin{equation*}
\begin{pmatrix}
G_u&G_v
\end{pmatrix}
\begin{pmatrix}
(c^2-v^2)[u]&(c^2-v^2)[v]\\
(2uv)[u]-(c^2-v^2)[v]&(c^2-u^2)[u]
\end{pmatrix}
\begin{pmatrix}
0&1\\
-1&0
\end{pmatrix}
///\begin{pmatrix}
y_u&y_v
\end{pmatrix}.
\end{equation*}
Thus, if $\qnt{\nabla y}^\circ$ satisfies \eqref{eq:yShockObliqueDerivativeU}, then the slope expression \eqref{eq:UpdateBoundaryCondition} is determined as follows
\begin{equation*}
\begin{split}
&\begin{pmatrix}
G_u&G_v
\end{pmatrix}
\begin{pmatrix}
(c^2-v^2)[u]&(c^2-v^2)[v]\\
(2uv)[u]-(c^2-v^2)[v]&(c^2-u^2)[u]
\end{pmatrix}
\begin{pmatrix}
0&1\\
-1&0
\end{pmatrix}\\
&\cdot
\begin{pmatrix}
(c^2+v^2)u&-(c^2-v^2)v\\
(c^2-u^2)v&(c^2-v^2)u
\end{pmatrix}
\begin{pmatrix}
0&1\\
-1&0
\end{pmatrix}
///\begin{pmatrix}
\dif u&\dif v
\end{pmatrix}.
\end{split}
\end{equation*}
Therefore
\begin{equation*}-
\begin{pmatrix}
G_u&G_v
\end{pmatrix}
\begin{pmatrix}
(c^2-v^2)\qnt{v[v]+u[u]}&(c^2+v^2)u[v]-(c^2-u^2)v[u]\\
(c^2+u^2)v[u]-(c^2-v^2)u[v]&(c^2-u^2)\qnt{v[v]+u[u]}
\end{pmatrix}
///\begin{pmatrix}
\frac{\dif u}{\dif v}&1
\end{pmatrix}.
\end{equation*}
To summarise, we can introduce the vector $s_g'$, which is defined via \eqref{eq:yWallPositionCondition} under the constrain \eqref{eq:yShockObliqueDerivativeUM}, as follows
\begin{equation}\label{eq:sg}
\begin{split}
s_g' :&=-\begin{pmatrix}
G_u&G_v
\end{pmatrix}
\begin{pmatrix}
(c^2-v^2)\qnt{v[v]+u[u]}&(c^2+v^2)u[v]-(c^2-u^2)v[u]\\
(c^2+u^2)v[u]-(c^2-v^2)u[v]&(c^2-u^2)\qnt{v[v]+u[u]}
\end{pmatrix}\\
&=-\begin{pmatrix}
G_u&G_v
\end{pmatrix}\qnt{\qnt{v[v]+u[u]}\begin{pmatrix}
c^2-v^2&uv\\
uv&c^2-u^2
\end{pmatrix}+c^2\qnt{v[u]-u[v]}\begin{pmatrix}
0&-1\\
1&0
\end{pmatrix}}\\
&=_{\epsilon=0}\bar{q}vc^2\begin{pmatrix}
G_v&-G_u
\end{pmatrix}.
\end{split}
\end{equation}

Similarly, consider the inner product between $s_g'$ and $\nabla_U G$, we have
\begin{equation}\label{eq:sg'DGexpression}
\begin{split}
&\quad s_g' \cdot \nabla_U G\\
&=-\begin{pmatrix}
G_u&G_v
\end{pmatrix}
\qnt{\qnt{v[v]+u[u]}\begin{pmatrix}
c^2-v^2&uv\\
uv&c^2-u^2
\end{pmatrix}+c^2\qnt{v[u]-u[v]}\begin{pmatrix}
0&-1\\
1&0
\end{pmatrix}}
\begin{pmatrix}
G_u\\G_v
\end{pmatrix}\\
&=
-\underbrace{\qnt{u[u]+v[v]}}_{\frac{\epsilon}{\rho}(u-\underline{u})}\begin{pmatrix}
G_u&G_v
\end{pmatrix}
\begin{pmatrix}
c^2-v^2&uv\\
uv&c^2-u^2
\end{pmatrix}
\begin{pmatrix}
G_u\\G_v
\end{pmatrix}>0\\
&=
\frac{\underline{u}-u}{\rho}\begin{pmatrix}
G_u&G_v
\end{pmatrix}
\begin{pmatrix}
c^2-v^2&uv\\
uv&c^2-u^2
\end{pmatrix}
\begin{pmatrix}
G_u\\G_v
\end{pmatrix}\epsilon>0\\
&=_{\epsilon=0}0
\end{split}
\end{equation}
where on the shock polar $G(U, \epsilon)=0$, subsonic part, we have
\begin{equation*}
\begin{pmatrix}
G_u&G_v
\end{pmatrix}
\begin{pmatrix}
c^2-v^2&uv\\
uv&c^2-u^2
\end{pmatrix}
\begin{pmatrix}
G_u\\G_v
\end{pmatrix}>0
\end{equation*}
and $u[u]+v[v]=\frac{\epsilon}{\rho}(u-\underline{u})<0$.

\subsection{$(k, \varrho)$ coordinates}Since we consider the case when $\epsilon$ is sufficiently small, focusing on solutions near \eqref{eq:ShockpolarS}, we need to introduce coordinates defined locally around \eqref{eq:ShockpolarS} that are more manageable. In the shock polar expression \eqref{eq:ShockPolarGExpressione}, we introduce the parameter $\varrho$ by letting $\epsilon=\varrho$. It is defined as follows
\begin{equation*}
\underline{\rho}(\varrho)\underline{u}(\varrho)=\varrho.
\end{equation*}
Here, for $\gamma=2$, let $\underline{\rho}\underline{u}=\varrho<<1$, we have
\begin{equation*}
\underline{\rho}\underline{u}=\frac{1}{4}\qnt{\bar{q}-\underline{u}}\qnt{\bar{q}+\underline{u}}\underline{u}=\varrho.
\end{equation*}
Fortunately, in the special case where $\gamma= 2$, owing to
\begin{equation*}
\partial_{\underline{u}}\qnt{\frac{1}{4}\qnt{\bar{q}-\underline{u}}\qnt{\bar{q}+\underline{u}}\underline{u}}|_{\underline{u}=\bar{q}}=-\frac{\bar{q}^2}{2},
\end{equation*}
by implicit function theorem, $\underline{u}$ and $\underline{\rho}$ are both $C^\infty$ functions of $\varrho$ around $\underline{u}=\bar{q}$.

The states $(u, v)$ along the shock should satisfy the shock polar
\begin{equation}\label{eq:ShockPolarGExpressionvarrho}
G(u, v, \varrho)=0,
\end{equation}
where $G$ is defined via \eqref{eq:ShockPolarGExpressione}. Thus, we can solve $\varrho$ uniquely from $(u, v)$ for $\varrho<<1$. To see this, we differentiate above with respect to $\varrho$ to have
\begin{equation}\label{eq:Ge}
G_\varrho=-\frac{\underline{u}-u}{\rho}-\underline{u}'\qnt{\frac{\varrho}{\rho}-u}.
\end{equation}
By entropy condition, we have
\begin{equation*}
\underline{u}-u>0.
\end{equation*}
By \eqref{eq:ShockPolarGExpressionvarrho}, we have
\begin{equation*}
u-\frac{\varrho}{\rho}=-\frac{v^2}{u-\underline{u}}>0.
\end{equation*}
It can be directly verified that
\begin{equation*}
\underline{u}'<0,
\end{equation*}
for $\varrho<<1$ and $\bar{q}-\underline{u}<<1$. Above three indicate $G_\varrho<0$, so $\varrho$ can be solved as a function of $(u, v)$ around \eqref{eq:ShockpolarS}.

Now $\varrho(u, v)$ is the implicit function defined by \eqref{eq:ShockPolarGExpressionvarrho} and we have
\begin{equation*}
\qnt{\varrho_u, \varrho_v}=-\frac{1}{G_\varrho}\qnt{G_u, G_v}.
\end{equation*}
Particular
\begin{equation}\label{eq:Ge0}
G_\varrho\Big|_{\varrho(u, v)=0}=\frac{4\bar{q}+2u}{\bar{q}^2}>0.
\end{equation}
\begin{lstlisting}
RestrictS[ImplicitD[G[u, v, U, e], InF[U, e] == 0, U, e]]
\end{lstlisting}

Next, the parameter $k$ is defined by
\begin{equation*}
k=\frac{u}{v}.
\end{equation*}
Thus
\begin{equation*}
\qnt{k_u, k_v}=\qnt{\frac{1}{v}, -\frac{u}{v^2}}.
\end{equation*}

To conclude, we have 
\begin{prop}
The coordinates $(k, \varrho)$ can be defined around \eqref{eq:ShockpolarS} in $(u, v)$ plane. The Jacobi matrix between $(k, \varrho)$ and $(u, v)$ can be calculated as follows
\begin{equation*}
\begin{pmatrix}
k_u&k_v\\
\varrho_u&\varrho_v
\end{pmatrix}
=
\begin{pmatrix}
\frac{1}{v}&-\frac{u}{v^2}\\
-\frac{G_u}{G_\varrho}&-\frac{G_v}{G_\varrho}
\end{pmatrix}
=_{\varrho=0}
\begin{pmatrix}
\frac{1}{v}&-\frac{u}{v^2}\\
\frac{\bar{q}^2 (\bar{q} - 2 u)}{2 (2 \bar{q} + u)}&-\frac{\bar{q}^2 v}{2\bar{q} + u}
\end{pmatrix},
\end{equation*}
and
\begin{equation*}
\begin{pmatrix}
u_k&u_\varrho\\
v_k&v_\varrho
\end{pmatrix}|_{\varrho=0}
=\begin{pmatrix}
k_u&k_v\\
\varrho_u&\varrho_v
\end{pmatrix}^{-1}|_{\varrho=0}=\frac{2(\bar{q}-u)(2\bar{q}+u)}{\bar{q}^3}
\begin{pmatrix}
\frac{\bar{q}^2 v}{2\bar{q} + u}&-\frac{u}{v^2}\\
\frac{\bar{q}^2 (\bar{q} - 2 u)}{2 (2 \bar{q} + u)}&-\frac{1}{v}
\end{pmatrix}.
\end{equation*}
Especially
\begin{equation}\label{eq:Partiale}
\begin{split}
\partial_\varrho&=u_\varrho\partial_u+v_\varrho\partial_v=_{\varrho=0}-\frac{2(\bar{q}-u)(2\bar{q}+u)}{\bar{q}^3v^2}\qnt{u\partial_u+v\partial_v}\\
&=_{\varrho=0}-\frac{2(2\bar{q}+u)}{\bar{q}^3u}\qnt{u\partial_u+v\partial_v}.
\end{split}
\end{equation}
\end{prop}

\subsection{Calculations of the slope free boundary value condition} We consider \eqref{eq:yWallPositionCondition} under $(k, \varrho)$ coordinates. It holds the differential relation
\begin{equation*}
\qnt{\dif k, \dif \varrho}=\qnt{\dif u, \dif v}
\begin{pmatrix}
k_u&\varrho_u\\
k_v&\varrho_v
\end{pmatrix}=\qnt{\dif u, \dif v}
\begin{pmatrix}
\frac{1}{v}&-\frac{G_u}{G_\varrho}\\
-\frac{u}{v^2}&-\frac{G_v}{G_\varrho}
\end{pmatrix}.
\end{equation*}
\begin{prop}
Owing to $s'///\qnt{\dif u, \dif v}$, by above, the free boundary condition \eqref{eq:yWallPositionCondition} can be rewritten under $(k, \varrho)$ coordinates as follows
\begin{equation}\label{eq:Fexpression}
\frac{\dif \varrho}{\dif k}=\frac{\varrho_u\dif u+\varrho_v\dif v}{k_u\dif u+k_v\dif v}=F(\frac{y_u}{y_v}(\varrho, k), \varrho, k):=-\frac{s' \cdot \nabla_U G}{G_\varrho \qnt{s' \cdot \nabla_U k}}.
\end{equation}
\end{prop}

For the incoming flow with mass flux $\epsilon$, on the shock polar $\varrho=\epsilon$, by \eqref{eq:yShockObliqueDerivativeU}, we have
\begin{equation}\label{eq:s=sg}
s'///s_g'.
\end{equation}
Thus, 
\begin{prop}
For the function $F(\frac{y_u}{y_v}(k, \varrho; \epsilon), k, \varrho)$, we have
\begin{equation}\label{eq:FexpressionG}
F(\frac{y_u}{y_v}(k, \epsilon; \epsilon), k, \epsilon)=-\frac{s' \cdot \nabla_U G}{G_\varrho \qnt{s' \cdot \nabla_U k}}\Big|_{\varrho=\epsilon}=-\frac{s_g' \cdot \nabla_U G}{G_\varrho \qnt{s_g' \cdot \nabla_U k}}\Big|_{\varrho=\epsilon},
\end{equation}
provided \eqref{eq:yShockObliqueDerivativeU}.
\end{prop}

To end this subsection, we present the following theorem about three calculations
\begin{thm}
\begin{equation}\label{eq:DeF'DGexpression}
\begin{split}
\partial_\varrho F(\frac{y_u}{y_v}\qnt{k, \varrho; \epsilon}, k, \varrho)\Big|_{\varrho=\epsilon=0; \nabla y=\nabla\Upsilon; \nabla^2 y=\nabla^2\Upsilon}
&=-\frac{\partial_\varrho\qnt{s' \cdot \nabla_U G}}{G_\varrho \qnt{s' \cdot \nabla_U k}}\Big|_{\varrho=\epsilon=0; \nabla y=\nabla\Upsilon; \nabla^2 y=\nabla^2\Upsilon}\\
&=-\frac{B''}{B'}-\frac{(2 \bar{q}^2 + 7 \bar{q} u - 12 u^2) v}{u (2 \bar{q} + u) (\bar{q} - 2 u)}<0\\
&=_{\sharp-k<\delta}-\frac{\alpha+1}{\sharp-k}-\frac{(2 \bar{q}^2 + 7 \bar{q} u - 12 u^2) v}{u (2 \bar{q} + u) (\bar{q} - 2 u)}<0,
\end{split}
\end{equation}
\begin{equation}\label{eq:DeFg'DGexpression}
\begin{split}
\partial_\varrho\frac{s_g' \cdot \nabla_U G}{G_\varrho \qnt{s_g' \cdot \nabla_U k}}\Big|_{\varrho=\epsilon=0}&=\frac{\partial_\varrho\qnt{s_g' \cdot \nabla_U G}}{G_\varrho \qnt{s_g' \cdot \nabla_U k}}\Big|_{\varrho=\epsilon=0}\\
&=\frac{2 (\bar{q} - 2 u) v}{u (2\bar{q} + u)}>0,
\end{split}
\end{equation}
and
\begin{equation}\label{eq:DeF-Fg'DGexpression}
\begin{split}
&\partial_\varrho F(\frac{y_u}{y_v}\qnt{k, \varrho; \epsilon}, k, \varrho)\Big|_{\varrho=\epsilon=0; \nabla y=\nabla\Upsilon; \nabla^2 y=\nabla^2\Upsilon}+ \partial_\varrho\frac{s_g' \cdot \nabla_U G}{G_\varrho \qnt{s_g' \cdot \nabla_U k}}\Big|_{\varrho=\epsilon=0}\\
&\quad=-\frac{\partial_\varrho\qnt{s' \cdot \nabla_U G}}{G_\varrho \qnt{s' \cdot \nabla_U k}}\Big|_{\varrho=\epsilon=0; \nabla y=\nabla\Upsilon; \nabla^2 y=\nabla^2\Upsilon}+\frac{\partial_\varrho\qnt{s_g' \cdot \nabla_U G}}{G_\varrho \qnt{s_g' \cdot \nabla_U k}}\Big|_{\varrho=0}\\
&\quad=-\frac{B''}{B'}-\frac{5 (3 \bar{q} - 4 u) v}{(\bar{q} - 2 u) (2 \bar{q} + u)}<0,
\end{split}
\end{equation}
provided that $(u, v)$ on \eqref{eq:ShockpolarS} satisfies
\begin{equation*}
0<\flat\leq \frac{u}{v}<\sharp<\sqrt{\frac{\gamma-1}{2}}=\frac{\sqrt{2}}{2},
\end{equation*}
or equivalently
\begin{equation*}
0<\flat_u\leq u<\sharp_u<\frac{\bar{q}}{3},\quad\mbox{and}\quad0<\flat_v\leq v<\sharp_v<\frac{\sqrt{2}\bar{q}}{3}.
\end{equation*}
\end{thm}

\begin{proof}
Since \eqref{eq:s=sg}, in view of the expression \eqref{eq:sg'DGexpression}, we have
\begin{equation}\label{eq:sg0}
\begin{cases}
s' \cdot \nabla_U G>0,~s_g' \cdot \nabla_U G>0\quad\mbox{for}\quad 0<\varrho=\epsilon;\\
s' \cdot \nabla_U G|_{\varrho=\epsilon=0}=s_g' \cdot \nabla_U G|_{\varrho=\epsilon=0}=0.
\end{cases}
\end{equation}
Then
\begin{equation}\label{eq:sgk}
\begin{split}
&\qnt{s_g' \cdot \nabla_U k}\\
&\quad=\frac{1}{v^2}\begin{pmatrix}
G_u&G_v
\end{pmatrix}
\qnt{\qnt{v[v]+u[u]}\begin{pmatrix}
c^2-v^2&uv\\
uv&c^2-u^2
\end{pmatrix}+c^2\qnt{v[u]-u[v]}\begin{pmatrix}
0&-1\\
1&0
\end{pmatrix}}
\begin{pmatrix}
-v\\
u
\end{pmatrix}\\
&\quad=_{\varrho=\epsilon=0}\frac{1}{v^2}\begin{pmatrix}
2u-\bar{q}&2v
\end{pmatrix}
\begin{pmatrix}
0&vc^2\bar{q}\\
-vc^2\bar{q}&0
\end{pmatrix}
\begin{pmatrix}
-v\\
u
\end{pmatrix}=\frac{uc^2\bar{q}^2}{v}=\frac{(\bar{q}-u)u\bar{q}^3}{2v},
\end{split}
\end{equation}
and
\begin{equation}\label{eq:sk}
\begin{split}
\qnt{s' \cdot \nabla_U k}&=\begin{pmatrix}
y_u&y_v
\end{pmatrix}
\begin{pmatrix}
(c^2-v^2)v&(c^2+v^2)u\\
-(c^2-v^2)u&(c^2-u^2)v
\end{pmatrix}
\begin{pmatrix}
k_u\\
k_v
\end{pmatrix}\\
&=\frac{1}{v^2}\begin{pmatrix}
y_u&y_v
\end{pmatrix}
\begin{pmatrix}
(c^2-v^2)v&(c^2+v^2)u\\
-(c^2-v^2)u&(c^2-u^2)v
\end{pmatrix}
\begin{pmatrix}
v\\
-u
\end{pmatrix}\\
&=\frac{1}{v^2}\begin{pmatrix}
y_u&y_v
\end{pmatrix}
\begin{pmatrix}
(c^2-v^2)v^2-(c^2+v^2)u^2\\
-(c^2-v^2)uv-(c^2-u^2)uv
\end{pmatrix}\\
&=\frac{1}{v^2}\qnt{\qnt{(c^2-v^2)v^2-(c^2+v^2)u^2}y_u-\qnt{2c^2-u^2-v^2}uvy_v}\\
&=_{\varrho=\epsilon=0; \nabla y=\nabla \Upsilon}\frac{B'\bar{q}^2u }{2 (\bar{q} - u) v}.
\end{split}
\end{equation}
\begin{lstlisting}
Simplify[
RestrictS[((CC[u, v] - v^2) v^2 - (CC[u, v] + v^2) u^2)/v^2] y0u - 
RestrictS[((2 CC[u, v] - u^2 - v^2) u  v)/v^2] y0v]
\end{lstlisting}
Next, for \eqref{eq:FexpressionG}, we have
\begin{equation}\label{eq:FexpressionG>}
F(\frac{y_u}{y_v}(k, \epsilon; \epsilon), k, \epsilon)=-\frac{s' \cdot \nabla_U G}{G_\varrho \qnt{s' \cdot \nabla_U k}}\Big|_{\varrho=\epsilon}=-\frac{s_g' \cdot \nabla_U G}{G_\varrho \qnt{s_g' \cdot \nabla_U k}}\Big|_{\varrho=\epsilon}=-g(k, \epsilon)<0,
\end{equation}
provided $\epsilon<<1$ and $\nabla y\approx\nabla\Upsilon$. Particular, for $\varrho=\epsilon=0$, we have
\begin{equation}\label{eq:FexpressionG0}
F(\frac{y_u}{y_v}(k, 0; 0), k, 0)=-\frac{s' \cdot \nabla_U G}{G_\varrho \qnt{s' \cdot \nabla_U k}}\Big|_{\varrho=\epsilon=0}=-\frac{s_g' \cdot \nabla_U G}{G_\varrho \qnt{s_g' \cdot \nabla_U k}}\Big|_{\varrho=\epsilon=0}=-g(k, 0)=0.
\end{equation}
By \eqref{eq:sg'DGexpression}
\begin{equation}\label{eq:Desg'DGexpression}
\begin{split}
\partial_\varrho\qnt{s_g' \cdot \nabla_U G}|_{\varrho=\epsilon=0}&=
\frac{\bar{q}-u}{\rho}\begin{pmatrix}
G_u&G_v
\end{pmatrix}
\begin{pmatrix}
c^2-v^2&uv\\
uv&c^2-u^2
\end{pmatrix}
\begin{pmatrix}
G_u\\G_v
\end{pmatrix}|_{\varrho=\epsilon=0}\\
&=2\bar{q}(\bar{q}-2u) (\bar{q}-u).
\end{split}
\end{equation}
\begin{lstlisting}
RestrictS[(Q - u)/
Rho[u, v] (Gu[u, v, U, e]^2 (CC[u, v] - v^2) + 
2 u  v  Gu[u, v, U, e] Gv[u, v, U, e] + 
Gv[u, v, U, e]^2 (CC[u, v] - u^2))]
\end{lstlisting}
Thus, by \eqref{eq:sg0}, we have
\begin{equation*}
\begin{split}
\partial_\varrho F(\frac{y_u}{y_v}\qnt{k, \varrho; \epsilon}, k, \varrho)\Big|_{\varrho=\epsilon=0; \nabla y=\nabla\Upsilon; \nabla^2 y=\nabla^2\Upsilon}&=\partial_\varrho\qnt{-\frac{s' \cdot \nabla_U G}{G_\varrho \qnt{s' \cdot \nabla_U k}}}\Big|_{\varrho=\epsilon=0; \nabla y=\nabla\Upsilon; \nabla^2 y=\nabla^2\Upsilon}\\
&=-\frac{\partial_\varrho\qnt{s' \cdot \nabla_U G}}{G_\varrho \qnt{s' \cdot \nabla_U k}}\Big|_{\varrho=\epsilon=0; \nabla y=\nabla\Upsilon; \nabla^2 y=\nabla^2\Upsilon},
\end{split}
\end{equation*}
and \eqref{eq:Desg'DGexpression}, \eqref{eq:Ge0} and \eqref{eq:sgk} lead to
\begin{equation*}
\begin{split}
\partial_\varrho\frac{s_g' \cdot \nabla_U G}{G_\varrho \qnt{s_g' \cdot \nabla_U k}}\Big|_{\varrho=\epsilon=0}&=\frac{\partial_\varrho\qnt{s_g' \cdot \nabla_U G}}{G_\varrho \qnt{s_g' \cdot \nabla_U k}}\Big|_{\varrho=\epsilon=0}\\
&=\frac{2 (\bar{q} - 2 u) v}{u (2\bar{q} + u)}.
\end{split}
\end{equation*}
\begin{lstlisting}
2  Q  (Q - 2  u)  (Q - u)/((2 (2 Q + u))/Q^2 (Q - u) u  Q^3/(2 v))
\end{lstlisting}

For the most complicate part, we calculate $\partial_\varrho\qnt{s' \cdot \nabla_U G}$. By the expression \eqref{eq:s'DG}, we have
\begin{equation}\label{eq:Des'DGexpression}
\begin{split}
&\partial_\varrho\qnt{s' \cdot \nabla_U G}|_{\varrho=\epsilon=0; \nabla y=\nabla\Upsilon; \nabla^2 y=\nabla^2\Upsilon}\\
&\quad=-\frac{2(2\bar{q}+u)}{\bar{q}^3u}\qnt{u\partial_u+v\partial_v}\qnt{s' \cdot \nabla_U G}|_{\varrho=\epsilon=0; \nabla y=\nabla\Upsilon; \nabla^2 y=\nabla^2\Upsilon}\\
&\quad=B''(k)\frac{\qnt{2\bar{q}+u}v}{\qnt{\bar{q}-u}^2}+B'(k)\frac{2\bar{q}^2 + 7 \bar{q} u - 12 u^2}{(\bar{q} - 2 u) (\bar{q} - u)}.
\end{split}
\end{equation}
\begin{lstlisting}
RestrictS[-2 (2 Q + 
u)/(Q^3 u) (u (ImplicitD[
v  (-v^2 + CC[u, v]) Gu[u, v, U, e] + 
u  (v^2 + CC[u, v]) Gv[u, v, U, e], {G[u, v, U, e] == 0, 
InF[U, e] == 0}, {U, e}, u] y0u + 
ImplicitD[
v  (-u^2 + CC[u, v]) Gv[u, v, U, e] - 
u  (-v^2 + CC[u, v]) Gu[u, v, U, e], {G[u, v, U, e] == 0, 
InF[U, e] == 0}, {U, e}, u] y0v +
(v  (-v^2 + CC[u, v]) Gu[u, v, U, e] + 
u  (v^2 + CC[u, v]) Gv[u, v, U, 
e]) y0uu + (v  (-u^2 + CC[u, v]) Gv[u, v, U, e] - 
u  (-v^2 + CC[u, v]) Gu[u, v, U, e]) y0uv) +
v (ImplicitD[
v  (-v^2 + CC[u, v]) Gu[u, v, U, e] + 
u  (v^2 + CC[u, v]) Gv[u, v, U, e], {G[u, v, U, e] == 0, 
InF[U, e] == 0}, {U, e}, v] y0u + 
ImplicitD[
v  (-u^2 + CC[u, v]) Gv[u, v, U, e] - 
u  (-v^2 + CC[u, v]) Gu[u, v, U, e], {G[u, v, U, e] == 0, 
InF[U, e] == 0}, {U, e}, v] y0v +
(v  (-v^2 + CC[u, v]) Gu[u, v, U, e] + 
u  (v^2 + CC[u, v]) Gv[u, v, U, 
e]) y0uv + (v  (-u^2 + CC[u, v]) Gv[u, v, U, e] - 
u  (-v^2 + CC[u, v]) Gu[u, v, U, e]) y0vv))]
\end{lstlisting}
So, above with \eqref{eq:Ge0} and \eqref{eq:sk}
\begin{equation*}
\begin{split}
\partial_\varrho F(\frac{y_u}{y_v}\qnt{k, \varrho; \epsilon}, k, \varrho)\Big|_{\varrho=\epsilon=0; \nabla y=\nabla\Upsilon; \nabla^2 y=\nabla^2\Upsilon}
&=-\frac{\partial_\varrho\qnt{s' \cdot \nabla_U G}}{G_\varrho \qnt{s' \cdot \nabla_U k}}\Big|_{\varrho=\epsilon=0; \nabla y=\nabla\Upsilon; \nabla^2 y=\nabla^2\Upsilon}\\
&=-\frac{B''}{B'}-\frac{(2 \bar{q}^2 + 7 \bar{q} u - 12 u^2) v}{u (2 \bar{q} + u) (\bar{q} - 2 u)},
\end{split}
\end{equation*}
\begin{lstlisting}
Simplify[-(((b  (Q - u)  (2  Q^2 + 7  Q  u - 12  u^2))/(Q - 2  u) + 
c  (2  Q + u)  v)/(Q - u)^2)/((b  Q^2  u)/(2  Q  v - 
2  u  v) (2  (2  Q + u))/Q^2)]
\end{lstlisting}
and
\begin{equation*}
\begin{split}
&-\frac{\partial_\varrho\qnt{s' \cdot \nabla_U G}}{G_\varrho \qnt{s' \cdot \nabla_U k}}\Big|_{\varrho=\epsilon=0; \nabla y=\nabla\Upsilon; \nabla^2 y=\nabla^2\Upsilon}+\frac{\partial_\varrho\qnt{s_g' \cdot \nabla_U G}}{G_\varrho \qnt{s_g' \cdot \nabla_U k}}\Big|_{\varrho=0}\\
&\quad=-\frac{B''}{B'}-\frac{5 (3 \bar{q} - 4 u) v}{(\bar{q} - 2 u) (2 \bar{q} + u)}.
\end{split}
\end{equation*}
\begin{lstlisting}
-(((b   (Q - u)   (2   Q^2 + 7   Q   u - 12   u^2))/(Q - 2   u) + 
c   (2   Q + u)   v)/(Q - u)^2)/((b   Q^2   u)/(2   Q   v - 
2   u   v)  (2   (2   Q + u))/Q^2) + 
2   Q   (Q - 
2   u)   (Q - u)/((2  (2  Q + u))/Q^2  (Q - u)  u   Q^3/(2  v))
\end{lstlisting}
The proof is complete.
\end{proof}

\subsection{Approximate slope free boundary function}
First, we introduce the approximate slope free boundary function $\mathcal{F}$ by extending $(\nabla y)^\circ$ from $\mathcal{S}$ via the derivatives of $\Upsilon$.
\begin{defn}
We define the approximate slope free boundary function $\mathcal{F}$ as follows
\begin{equation*}
\begin{split}
\mathcal{F}(\Upsilon, k, \varrho; \epsilon)&=\partial_\varrho F(\frac{y_u}{y_v}\qnt{k, \varrho; \epsilon}, k, \varrho)\Big|_{\varrho=\epsilon=0; \nabla y=\nabla\Upsilon; \nabla^2 y=\nabla^2\Upsilon}\qnt{\varrho-\epsilon} \\
&\quad- \frac{s_g' \cdot \nabla_U G}{G_\varrho \qnt{s_g' \cdot \nabla_U k}}\Big|_{\varrho=\epsilon}.
\end{split}
\end{equation*}
Then the approximate boundary $\varrho=\Gamma(k)$ is determined via the following ODE 
\begin{equation}\label{eq:ApproximateUpdateODE}
\begin{split}
\frac{\dif \Gamma}{\dif k}=\mathcal{F}(\Upsilon, k, \Gamma; \epsilon)&=\partial_\varrho F(\frac{y_u}{y_v}\qnt{k, \varrho; \epsilon}, k, \varrho)\Big|_{\varrho=\epsilon=0; \nabla y=\nabla\Upsilon; \nabla^2 y=\nabla^2\Upsilon}\qnt{\Gamma-\epsilon} \\
&\quad- \frac{s_g' \cdot \nabla_U G}{G_\varrho \qnt{s_g' \cdot \nabla_U k}}\Big|_{\varrho=\epsilon},
\end{split}
\end{equation}
with the initial value condition $\Gamma(\flat)=\epsilon$.
\end{defn}
For this $\Gamma$, we have
\begin{thm}\label{thm:GammaProperties}
There exists $\epsilon_\Gamma>0$, such that for $\epsilon<\epsilon_\Gamma$, $\Gamma$ can be solved in $[\flat, \sharp]$, and we have
\begin{equation}\label{eq:GammaDomain}
0\leq\Gamma(k)\leq\epsilon,\quad\mbox{and}\quad\Gamma(\flat)=\Gamma(\sharp)=0,
\end{equation}
\begin{equation}\label{eq:GammaEstimates}
\norm{\Gamma}_{C^{2, \alpha}}\leq C\epsilon,
\end{equation}
and
\begin{equation}\label{eq:GammalowerEstimate}
\inf_{k\in[\flat, \sharp)}\abs{\frac{\Gamma(k)-\epsilon}{\epsilon\sin(\pi\frac{k-\flat}{\sharp-\flat})}}\geq C.
\end{equation}
\end{thm}
\begin{proof}
By \eqref{eq:sg0}, we can rewrite $\mathcal{F}(\Upsilon, k, \varrho; \epsilon)$ as follows
\begin{equation*}
\begin{split}
\mathcal{F}(\Upsilon, k, \varrho; \epsilon)
&=\partial_\varrho F(\frac{y_u}{y_v}\qnt{k, \varrho; \epsilon}, k, \varrho)\Big|_{\varrho=\epsilon=0; \nabla y=\nabla\Upsilon; \nabla^2 y=\nabla^2\Upsilon}\qnt{\varrho-\epsilon}\\
&\quad - \partial_\varrho\frac{s_g' \cdot \nabla_U G}{G_\varrho \qnt{s_g' \cdot \nabla_U k}}\Big|_{\varrho=0}\epsilon\\
&\quad- \underbrace{\qnt{\frac{s_g' \cdot \nabla_U G}{G_\varrho \qnt{s_g' \cdot \nabla_U k}}\Big|_{\varrho=\epsilon}-\frac{s_g' \cdot \nabla_U G}{G_\varrho \qnt{s_g' \cdot \nabla_U k}}\Big|_{\varrho=0} -\partial_\varrho\frac{s_g' \cdot \nabla_U G}{G_\varrho \qnt{s_g' \cdot \nabla_U k}}\Big|_{\varrho=0}\epsilon}}_{\leq C\epsilon^2}.
\end{split}
\end{equation*}
When $\varrho=\epsilon$, by \eqref{eq:DeFg'DGexpression}
\begin{equation*}
\begin{split}
\mathcal{F}(\Upsilon, k, \epsilon; \epsilon)&=- \partial_\varrho\frac{s_g' \cdot \nabla_U G}{G_\varrho \qnt{s_g' \cdot \nabla_U k}}\Big|_{\varrho=0}\epsilon\\
&\quad- \qnt{\frac{s_g' \cdot \nabla_U G}{G_\varrho \qnt{s_g' \cdot \nabla_U k}}\Big|_{\varrho=\epsilon}-\frac{s_g' \cdot \nabla_U G}{G_\varrho \qnt{s_g' \cdot \nabla_U k}}\Big|_{\varrho=0}-\partial_\varrho\frac{s_g' \cdot \nabla_U G}{G_\varrho \qnt{s_g' \cdot \nabla_U k}}\Big|_{\varrho=0}\epsilon}\\
&=-\frac{\partial_\varrho\qnt{s_g' \cdot \nabla_U G}}{G_\varrho \qnt{s_g' \cdot \nabla_U k}}\Big|_{\varrho=0}\epsilon\\
&\quad- \qnt{\frac{s_g' \cdot \nabla_U G}{G_\varrho \qnt{s_g' \cdot \nabla_U k}}\Big|_{\varrho=\epsilon}-\frac{s_g' \cdot \nabla_U G}{G_\varrho \qnt{s_g' \cdot \nabla_U k}}\Big|_{\varrho=0}-\partial_\varrho\frac{s_g' \cdot \nabla_U G}{G_\varrho \qnt{s_g' \cdot \nabla_U k}}\Big|_{\varrho=0}\epsilon}\\
&=-\frac{2 (\bar{q} - 2 u) v}{u (2\bar{q} + u)}\epsilon\\
&\quad\underbrace{-\qnt{\frac{s_g' \cdot \nabla_U G}{G_\varrho \qnt{s_g' \cdot \nabla_U k}}\Big|_{\varrho=\epsilon}-\frac{s_g' \cdot \nabla_U G}{G_\varrho \qnt{s_g' \cdot \nabla_U k}}\Big|_{\varrho=0}-\partial_\varrho\frac{s_g' \cdot \nabla_U G}{G_\varrho \qnt{s_g' \cdot \nabla_U k}}\Big|_{\varrho=0}\epsilon}}_{\leq C\epsilon^2}\\
&\leq-\frac{ (\bar{q} - 2 u) v}{u (2\bar{q} + u)}\epsilon<0,
\end{split}
\end{equation*}
and when $\varrho=0$, by \eqref{eq:DeF-Fg'DGexpression}
\begin{equation*}
\begin{split}
\mathcal{F}(\Upsilon, k, 0; \epsilon)
&=-\epsilon\qnt{-\frac{\partial_\varrho\qnt{s' \cdot \nabla_U G}}{G_\varrho \qnt{s' \cdot \nabla_U k}}\Big|_{\varrho=\epsilon=0; \nabla y=\nabla\Upsilon; \nabla^2 y=\nabla^2\Upsilon}+\frac{\partial_\varrho\qnt{s_g' \cdot \nabla_U G}}{G_\varrho \qnt{s_g' \cdot \nabla_U k}}\Big|_{\varrho=0}}\\
&\quad- \qnt{\frac{s_g' \cdot \nabla_U G}{G_\varrho \qnt{s_g' \cdot \nabla_U k}}\Big|_{\varrho=\epsilon}-\frac{s_g' \cdot \nabla_U G}{G_\varrho \qnt{s_g' \cdot \nabla_U k}}\Big|_{\varrho=0}-\partial_\varrho\frac{s_g' \cdot \nabla_U G}{G_\varrho \qnt{s_g' \cdot \nabla_U k}}\Big|_{\varrho=0}\epsilon}\\
&=-\epsilon\qnt{-\frac{B''}{B'}-\frac{5 (3 \bar{q} - 4 u) v}{(\bar{q} - 2 u) (2 \bar{q} + u)}}\\
&\quad\underbrace{- \qnt{\frac{s_g' \cdot \nabla_U G}{G_\varrho \qnt{s_g' \cdot \nabla_U k}}\Big|_{\varrho=\epsilon}-\frac{s_g' \cdot \nabla_U G}{G_\varrho \qnt{s_g' \cdot \nabla_U k}}\Big|_{\varrho=0}-\partial_\varrho\frac{s_g' \cdot \nabla_U G}{G_\varrho \qnt{s_g' \cdot \nabla_U k}}\Big|_{\varrho=0}\epsilon}}_{\leq C\epsilon^2}\\
&\geq \frac{\epsilon}{2}\qnt{\frac{B''}{B'}+\frac{5 (3 \bar{q} - 4 u) v}{(\bar{q} - 2 u) (2 \bar{q} + u)}}>0\\
&=_{\sharp-k<\delta}\frac{\epsilon}{2}\qnt{\frac{\alpha+1}{\sharp-k}+\frac{5 (3 \bar{q} - 4 u) v}{(\bar{q} - 2 u) (2 \bar{q} + u)}}>0.
\end{split}
\end{equation*}
As illustrated in Figure \ref{fig:UpdateVectors} and Figure \ref{fig:UpdateVectorsCorner}, above two indicate \eqref{eq:GammaDomain} for $\epsilon<<1$.
\begin{figure}[htbp]
	\setlength{\unitlength}{1bp}
	\begin{center}
		\begin{picture}(190,190)
		\put(-30,-10){\includegraphics[scale=0.4]{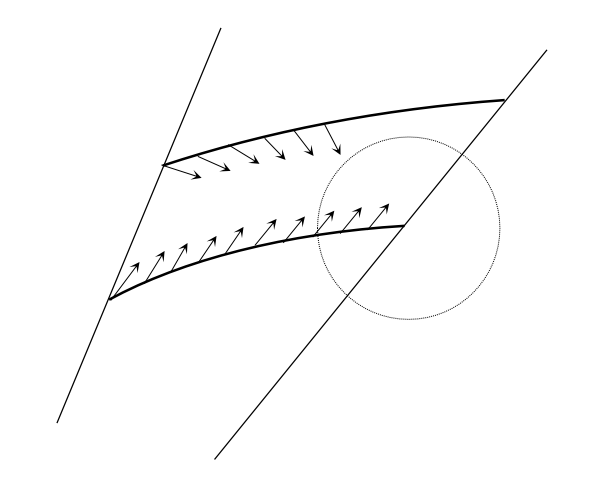}}
		\put(32,135){\tiny$\frac{u}{v}=\flat$}
		\put(130,130){\tiny$\frac{u}{v}=\sharp$}
		\put(-20,88){\tiny$\mathcal{S}:\varrho=0$}
		\put(-56,48){\tiny$\mathcal{S}:\varrho=\epsilon<<1$}
		\end{picture}
	\end{center}
	\caption{update vectors}\label{fig:UpdateVectors}
\end{figure}
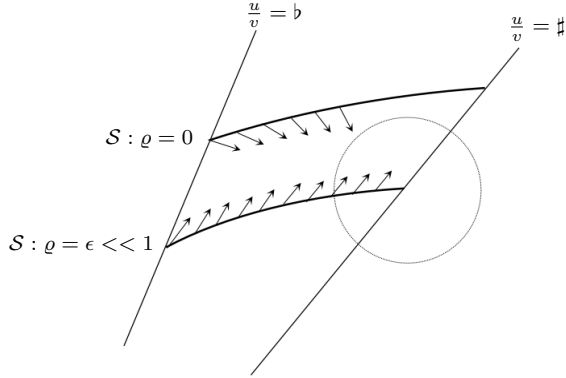
\begin{figure}[htbp]
	\setlength{\unitlength}{1bp}
	\begin{center}
		\begin{picture}(190,190)
		\put(-20,-10){\includegraphics[scale=0.3]{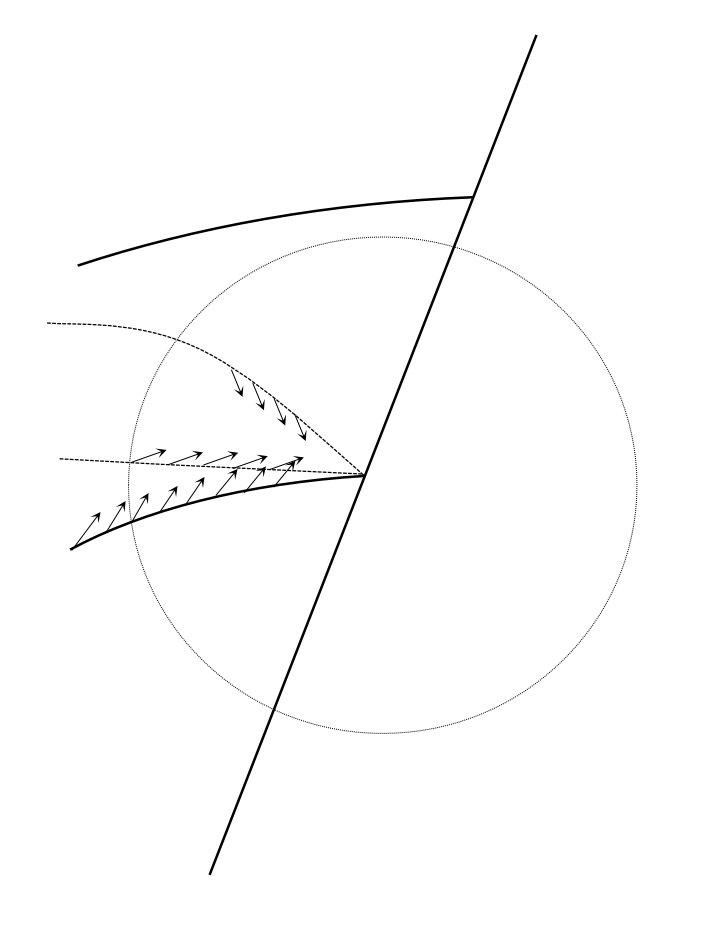}}
		\put(105,190){\tiny$\frac{u}{v}=\sharp$}
		\put(-40,135){\tiny$\mathcal{S}:\varrho=0$}
		\put(-60,71){\tiny$\mathcal{S}:\varrho=\epsilon<<1$}
		\end{picture}
	\end{center}
	\caption{update vectors around infinity flow field}\label{fig:UpdateVectorsCorner}
\end{figure}
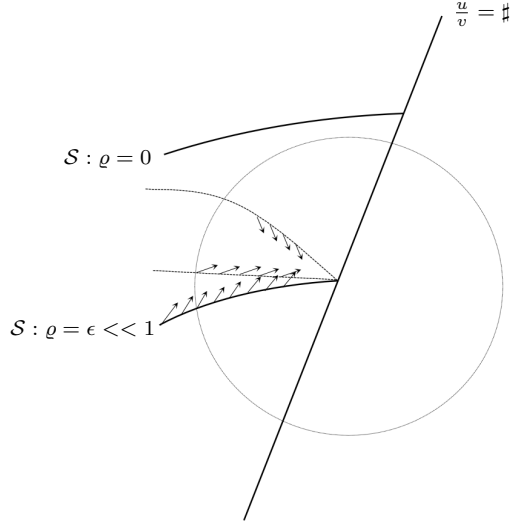

Then, by the expression \eqref{eq:sg'DGexpression}, we set
\begin{equation*}
\frac{s_g' \cdot \nabla_U G}{G_\varrho \qnt{s_g' \cdot \nabla_U k}}\Big|_{\varrho=\epsilon}=\epsilon \frac{s_g' \cdot \nabla_U G}{\varrho G_\varrho \qnt{s_g' \cdot \nabla_U k}}\Big|_{\varrho=\epsilon}=:\epsilon h(k, \epsilon)>0.
\end{equation*}
Therefore, $h\in C^\infty$ around \eqref{eq:ShockpolarS} and
\begin{equation*}
h(k, 0)=\partial_\varrho\frac{s_g' \cdot \nabla_U G}{G_\varrho \qnt{s_g' \cdot \nabla_U k}}\Big|_{\varrho=\epsilon=0}=\frac{2 (\bar{q} - 2 u) v}{u (2\bar{q} + u)}>0.
\end{equation*}
As a result, we can rewrite \eqref{eq:ApproximateUpdateODE} as follows
\begin{equation}\label{eq:GammaODEk}
\begin{split}
\frac{\dif \Gamma}{\dif k}&=\qnt{-\frac{B''}{B'}-\underbrace{\frac{(2 \bar{q}^2 + 7 \bar{q} u - 12 u^2) v}{u (2 \bar{q} + u) (\bar{q} - 2 u)}}_{H(k)>0}}\qnt{\Gamma-\epsilon}- \epsilon h(k, \epsilon)\\
&=\qnt{-\qnt{\ln B'(k)}'-H(k)}\qnt{\Gamma-\epsilon}- \epsilon h(k, \epsilon).
\end{split}
\end{equation}
For \eqref{eq:GammaODEk}, we have
\begin{equation}\label{eq:GammaExpression}
\begin{split}
\frac{\dif \qnt{\Gamma-\epsilon}}{\dif k}+\qnt{\qnt{\ln B'(k)}'+H(k)}\qnt{\Gamma-\epsilon}&=- \epsilon h(k, \epsilon)\\
\frac{\dif }{\dif k}\qnt{B'(k)\me^{\int_{\flat}^{k}H(\tau)\dif\tau}\qnt{\Gamma-\epsilon}}&=- \epsilon B'(k)\me^{\int_{\flat}^{k}H(\tau)\dif\tau}h(k, \epsilon)\\
B'(k)\me^{\int_{\flat}^{k}H(\tau)\dif\tau}\qnt{\Gamma-\epsilon}&=-\epsilon\int_{\flat}^{k} B'(t)\me^{\int_{\flat}^{t}H(\tau)\dif\tau}h(t, \epsilon)\dif t\\
\qnt{\Gamma-\epsilon}&=-\frac{\epsilon}{B'(k)}\int_{\flat}^{k} B'(t)\me^{-\int_{t}^{k}H(\tau)\dif\tau}h(t, \epsilon)\dif t<0.
\end{split}
\end{equation}

To consider the regularity around $k=\sharp$, we apply \eqref{eq:Bassumptions} to have
\begin{equation*}
\qnt{\Gamma-\epsilon}\overset{\sharp-k<\delta}{=}-\frac{\epsilon}{\alpha}(\sharp-k)^{\alpha+1}\int_{\flat}^{k} B'(t)\me^{-\int_{t}^{k}H(\tau)\dif\tau}h(t, \epsilon)\dif t.
\end{equation*}
Integrate by part,
\begin{equation}\label{eq:GammaExpressionSharp}
\begin{split}
\qnt{\Gamma-\epsilon}&=-\frac{\epsilon}{\alpha}(\sharp-k)^{\alpha+1}\bigg(B(t)\me^{-\int_{t}^{k}H(\tau)\dif\tau}h(t, \epsilon)\big|_\flat^k\\
&\quad-\int_{\flat}^{k}B(t)\qnt{H(t)\me^{-\int_{t}^{k}H(\tau)\dif\tau}h(t, \epsilon)+\me^{-\int_{t}^{k}H(\tau)\dif\tau}h_1(t, \epsilon)}\dif t\bigg)\\
&=-\frac{\epsilon}{\alpha}(\sharp-k)^{\alpha+1}\bigg(B(k)h(k, \epsilon)-B(\flat)\me^{-\int_{\flat}^{k}H(\tau)\dif\tau}h(\flat, \epsilon)\\
&\quad-\int_{\flat}^{k}B(t)\qnt{H(t)\me^{-\int_{t}^{k}H(\tau)\dif\tau}h(t, \epsilon)+\me^{-\int_{t}^{k}H(\tau)\dif\tau}h_1(t, \epsilon)}\dif t\bigg)\\
&=-\frac{\epsilon}{\alpha}(\sharp-k)^{\alpha+1}\bigg((\sharp-k)^{-\alpha}h(k, \epsilon)-B(\flat)\me^{-\int_{\flat}^{k}H(\tau)\dif\tau}h(\flat, \epsilon)\\
&\quad-\int_{\flat}^{k}B(t)\qnt{H(t)\me^{-\int_{t}^{k}H(\tau)\dif\tau}h(t, \epsilon)+\me^{-\int_{t}^{k}H(\tau)\dif\tau}h_1(t, \epsilon)}\dif t\bigg)\\
&=-\frac{\epsilon}{\alpha}(\sharp-k)h(k, \epsilon)+\frac{\epsilon}{\alpha}(\sharp-k)^{\alpha+1}B(\flat)\me^{-\int_{\flat}^{k}H(\tau)\dif\tau}h(\flat, \epsilon)\\
&\quad+\frac{\epsilon}{\alpha}(\sharp-k)^{\alpha+1}\int_{\flat}^{k}B(t)\qnt{H(t)h(t, \epsilon)+h_t(t, \epsilon)}\me^{-\int_{t}^{k}H(\tau)\dif\tau}\dif t.
\end{split}
\end{equation}
Since $B\in C^\infty$, by \eqref{eq:GammaExpression}, \eqref{eq:GammaEstimates} holds for $k\in[\flat, \sharp-\frac{\delta}{2}]$. Then, for the interval $[\sharp-\delta, \sharp]$ part, we apply \eqref{eq:GammaExpressionSharp} to show \eqref{eq:GammaEstimates}. To prove \eqref{eq:GammalowerEstimate}, by \eqref{eq:GammaExpressionSharp}, we have
\begin{equation*}
\lim\limits_{k\rightarrow\sharp^+}\frac{\abs{\Gamma(k)-\epsilon}}{\epsilon\sin(\pi\frac{k-\flat}{\sharp-\flat})}=\frac{\sharp-\flat}{\alpha\pi}h(\sharp, \epsilon).
\end{equation*}
Hence, there exists a positive constant $\Delta\in(0, \delta)$, such that
\begin{equation*}
\inf_{k\in[\sharp-\Delta, \sharp]}\abs{\frac{\Gamma(k)-\epsilon}{\epsilon\sin(\pi\frac{k-\flat}{\sharp-\flat})}}>\frac{\sharp-\flat}{2\alpha\pi}h(\sharp, 0),
\end{equation*}
for $\epsilon<\epsilon_\sharp$. For the $[\flat, \sharp-\frac{\Delta}{2}]$ part, in view of \eqref{eq:GammaExpression}, we have
\begin{equation*}
\begin{split}
\inf_{k\in[\flat, \sharp-\frac{\Delta}{2}]}\abs{\frac{\Gamma(k)-\epsilon}{\epsilon\sin(\pi\frac{k-\flat}{\sharp-\flat})}}&\geq\inf_{k\in[\flat, \sharp-\frac{\Delta}{2}]} \frac{1}{\epsilon\pi\frac{k-\flat}{\sharp-\flat}}\frac{\epsilon}{B'(k)}\int_{\flat}^{k} B'(t)\me^{-\int_{t}^{k}H(\tau)\dif\tau}h(t, \epsilon)\dif t\\
&\geq\frac{\sharp-\flat}{2\pi }\inf_{k\in[\flat, \sharp-\frac{\Delta}{2}]}\frac{\int_{\flat}^{k} B'(t)\me^{-\int_{t}^{k}H(\tau)\dif\tau}h(t, 0)\dif t}{B'(k)(k-\flat)},
\end{split}
\end{equation*}
for $\epsilon<\epsilon_\flat$. Combine above two, we have \eqref{eq:GammalowerEstimate}.
\end{proof}

The domain bounded by $\set{(k, \varrho)|\varrho=\Gamma(k)}$ and $\mathcal{S}$ is denoted by $\Omega_\Gamma$. We define the following weighted norms
\begin{equation*}
\begin{split}
\norm{\varphi}_{C^{(-m)}_{p, \beta}([\flat, \sharp))}&=\max\limits_{i=0,...,p}\sup\limits_{k\in[\flat, \sharp)}\abs{(\sharp-k)^{m+i}\varphi^{(i)}(k)}\\
&\quad+\sup\limits_{\flat\leq k_1< k_2<\sharp}\abs{(\sharp-k_2)^{m+p+\beta}\frac{\varphi^{(p)}(k_1)-\varphi^{(p)}(k_2)}{\abs{k_1-k_2}^\beta}},
\end{split}
\end{equation*}
\begin{equation*}
\begin{split}
\norm{u}_{C^{(-m)}_{p, \beta}(\Omega_\Gamma)}&=\max\limits_{i=0,...,p}\sup\limits_{(k, \varrho)\in\Omega}\abs{\qnt{\sharp-k}^{m+i}\nabla^i u(k, \varrho)}\\
&\quad+\sup\limits_{\flat\leq k_1\leq k_2<\sharp}\abs{\qnt{\sharp-k_2}^{m+p+\beta}\frac{\nabla^{p} u(k_1, \varrho_1)-\nabla^{p} u(k_2, \varrho_2)}{\abs{(k_1, \varrho_1)-(k_2, \varrho_2)}^\beta}},
\end{split}
\end{equation*}
and
\begin{equation*}
\begin{split}
\norm{f}_{C^{(1+\beta)}_{2, \beta}([\flat, \sharp))}&=\max\limits_{i=0, 1, 2}\sup\limits_{k\in[\flat, \sharp)}\abs{(\sharp-k)f^{(i)}(k)}+\sup\limits_{\flat\leq k_1< k_2<\sharp}\abs{(\sharp-k_2)\frac{f^{(2)}(k_1)-f^{(2)}(k_2)}{\abs{k_1-k_2}^\beta}}.
\end{split}
\end{equation*}
By Theorem \ref{thm:GammaProperties}, $\Omega_\Gamma$ collapses to the curve segment \eqref{eq:ShockpolarS} corresponding to $\frac{u}{v} \in [\flat, \sharp)$. Owing to \eqref{eq:SupersonicWallCondition}, this part is subsonic and the linear equation \eqref{eq:ySecondOrderEquation} is uniform elliptic. Thus, for $\epsilon<<1$, \eqref{eq:ySecondOrderEquation} is elliptic in $\Omega_\Gamma$ and we have
\begin{thm}
For this $\Gamma$, we consider equation \eqref{eq:ySecondOrderEquation} with the mixed boundary value condition \eqref{eq:yWallVelocityCondition} and \eqref{eq:yShockObliqueDerivativeU} in $\Omega_\Gamma$. It admits a solution $y^\Gamma$ and holds
\begin{equation}\label{eq:yGammaEstimate}
\norm{y^\Gamma}_{C_{2, \beta}^{(-\alpha)}(\Omega)}\le \bar{C}\norm{B}_{C_{2, \beta}^{(-\alpha)}([\flat, \sharp))}.
\end{equation}
\end{thm}
\begin{proof}
$B$ is unbounded. To solve it, let $B_n\in C^\infty$ 
\begin{equation*}
0<B_{n-1}\leq B_n\rightarrow B,
\end{equation*}
\begin{equation*}
B_n|_{[\flat, \sharp-\frac{1}{n}]}=B|_{[\flat, \sharp-\frac{1}{n}]},
\end{equation*}
and
\begin{equation*}
B_n|_{[\sharp-\frac{1}{2n}, \sharp]}=B(\sharp-\frac{1}{2n})=(2n)^\alpha.
\end{equation*}
Substituting $B$ by $B_n$, we solve \eqref{eq:ySecondOrderEquation} with 
\begin{equation}\label{eq:yWallVelocityConditionApproximate}
y^\Gamma_n=B_n\quad\mbox{on}\quad \varrho=\Gamma(k),
\end{equation}
and \eqref{eq:yShockObliqueDerivativeU} to obtain the solution $y^\Gamma_n$ by \cite{Lieberman2013WSPCH}. Then by maximum principle, we have
\begin{equation*}
0<y^\Gamma_{n-1}<y^\Gamma_n.
\end{equation*}
Let
\begin{equation*}
\lim_{n\rightarrow\infty} y^\Gamma_n=y^\Gamma.
\end{equation*}
$y^\Gamma$ is the solution. The estimate \eqref{eq:yGammaEstimate} is a conclusion of \cite{HuHuangArxiv2024}.
\end{proof}

We have the following approximate estimates
\begin{thm}\label{thm:PerturbationEstimates}
We have
\begin{equation}\label{eq:PerturbationEstimate2}
\abs{\nabla^2 y^\Gamma(k, \varrho)-\nabla^2\Upsilon(k)}\leq C\epsilon^\beta(\sharp-k)^{-\alpha-2}\norm{B}_{C_{2, \beta}^{(-\alpha)}([\flat, \sharp))},
\end{equation}
and
\begin{equation}\label{eq:PerturbationEstimate01}
\abs{\nabla^i y^\Gamma(k, \varrho)-\nabla^i\Upsilon(k)}\leq C\epsilon(\sharp-k)^{-\alpha-i}\norm{B}_{C_{2, \beta}^{(-\alpha)}([\flat, \sharp))}\quad\mbox{for}\quad i=0,1.
\end{equation}
\end{thm}
\begin{proof}
Consider $y^\Gamma(k, \varrho)-\Upsilon(k)$, we have
\begin{equation*}
\begin{split}
\abs{y^\Gamma(k, \varrho)-\Upsilon(k)}&=\abs{y^\Gamma(k, \varrho)-y^\Gamma(k, \Gamma(k))}\\
&\leq \max\limits_{\Gamma\leq\varrho\leq\epsilon}\abs{y_\varrho^\Gamma(k, \varrho)}\abs{\epsilon-\Gamma}\\
&\leq C\epsilon (\sharp-k)^{-\alpha}\norm{B}_{C_{2, \beta}^{(-\alpha)}([\flat, \sharp))}.
\end{split}
\end{equation*}
For $\nabla y^\Gamma(k, \varrho)-\nabla \Upsilon(k)$, differentiate the equation
\begin{equation*}
y^\Gamma(k, \Gamma(k))=B(k),
\end{equation*}
we have
\begin{equation}\label{eq:yGammaB1}
y^\Gamma_\varrho(k, \Gamma(k))\Gamma'(k)+y^\Gamma_k(k, \Gamma(k))=B'(k)\quad\mbox{on}\quad \varrho=\Gamma(k).
\end{equation}
For $\Upsilon$, we have
\begin{equation*}
\Upsilon(k)=B(k),
\end{equation*}
and \eqref{eq:WallConditionBS1} indicates the following formal derivative expression
\begin{equation}\label{eq:UpsilonB1}
\Upsilon_k=\Upsilon_u(k)u_k(k, 0)+\Upsilon_v(k)v_k(k, 0)=B'(k).
\end{equation}
For the convenience of typesetting, we introduce following notation
\begin{equation}\label{eq:W1Setting}
W^1_1(l, \varsigma)D^1y(k, \varrho):=u_k(l, \varsigma)y_u(k, \varrho)+v_k(l, \varsigma)y_v(k, \varrho).
\end{equation}

Then, \eqref{eq:yGammaB1} can be rewritten as
\begin{equation*}
W^1_1(k, \Gamma(k))D^1y^\Gamma(k, \Gamma(k))=y^\Gamma_k(k, \Gamma(k))=B'(k)-y^\Gamma_\varrho(k, \Gamma(k))\Gamma'(k).
\end{equation*}
$W^1_1(k, 0)D^1y^\Gamma(k, \varrho)$ can be expressed
\begin{equation*}
\begin{split}
&W^1_1(k, 0)D^1y^\Gamma(k, \varrho)\\
&\quad=B'(k)-y^\Gamma_\varrho(k, \Gamma(k))\Gamma'(k)+W^1_1(k, 0)D^1y^\Gamma(k, \varrho)-W^1_1(k, \Gamma(k))D^1y^\Gamma(k, \Gamma(k)).
\end{split}
\end{equation*}
Meanwhile, \eqref{eq:UpsilonB1} is reduced to
\begin{equation*}
W^1_1(k, 0)D^1\Upsilon(k)=B'(k).
\end{equation*}
Thus combine above two, we deduce
\begin{equation*}
\begin{split}
&W^1_1(k, 0)\qnt{D^1y^\Gamma(k, \varrho)-D^1\Upsilon(k)}\\
&\quad=-y^\Gamma_\varrho(k, \Gamma(k))\Gamma'(k)+W^1_1(k, 0)D^1y^\Gamma(k, \varrho)-W^1_1(k, \Gamma(k))D^1y^\Gamma(k, \Gamma(k)).
\end{split}
\end{equation*}
In view of Theorem \ref{thm:GammaProperties}, we have
\begin{equation*}
\abs{y^\Gamma_\varrho(k, \Gamma(k))\Gamma'(k)}\leq C\epsilon(\sharp-k)^{-\alpha-1}
\end{equation*}
and
\begin{equation*}
\begin{split}
&\abs{W^1_1(k, 0)D^1y^\Gamma(k, \varrho)-W^1_1(k, \Gamma(k))D^1y^\Gamma(k, \Gamma(k))}\\
&\quad\leq \abs{W^1_1(k, 0)}\abs{D^1y^\Gamma(k, \varrho)-D^1y^\Gamma(k, \Gamma(k))}\\
&\quad\quad+\abs{D^1y^\Gamma(k, \Gamma(k))}\abs{W^1_1(k, 0)-W^1_1(k, \Gamma(k))}\\
&\quad\leq C\abs{\Gamma-\epsilon}\abs{D^2y^\Gamma}+ C\epsilon \abs{D^1y^\Gamma}\\
&\quad\leq C\epsilon(\sharp-k)^{-\alpha-1}.
\end{split}
\end{equation*}
To summarise, we have
\begin{equation}\label{eq:W1Estimate}
\abs{W^1_1(k, 0)\qnt{D^1y^\Gamma(k, \varrho)-D^1\Upsilon(k)}}\leq C\epsilon(\sharp-k)^{-\alpha-1}.
\end{equation}

In \eqref{eq:yShockObliqueDerivativeU}, we set
\begin{equation}\label{eq:DefIJ}
\begin{cases}
I:=G_u(c^2-v^2)[u]-G_v(c^2-v^2)[v]+G_v[u](2uv);\\
J:=G_u(c^2-v^2)[v]+G_v(c^2-u^2)[u].
\end{cases}
\end{equation}
Then
\begin{equation}\label{eq:yGammaS1}
I(k, \epsilon; \epsilon)y_u^\Gamma(k, \epsilon)+J(k, \epsilon; \epsilon)y_v^\Gamma(k, \epsilon)=0.
\end{equation}
For \eqref{eq:ShockPolarObliqueConditionS}, we have
\begin{equation}\label{eq:UpsilonS1}
I(k, 0; 0)\Upsilon_u+J(k, 0; 0)\Upsilon_v=0.
\end{equation}
As \eqref{eq:W1Setting}, we set
\begin{equation*}
S^1_1(l, \varsigma)D^1y(k, \varrho)=I(l, \varsigma; \varsigma)y_u(k, \varrho)+J(l, \varsigma; \varsigma)y_v(k, \varrho)
\end{equation*}
Then \eqref{eq:yGammaS1} and \eqref{eq:UpsilonS1} are rewritten as follows
\begin{equation*}
S^1_1(k, \epsilon)D^1y^\Gamma(k, \epsilon)=0
\end{equation*}
and
\begin{equation*}
S^1_1(k, 0)D^1\Upsilon(k)=0.
\end{equation*}
Thus
\begin{equation*}
S^1_1(k, 0)\qnt{D^1y^\Gamma(k, \varrho)-D^1\Upsilon(k)}=S^1_1(k, 0)D^1y^\Gamma(k, \varrho)-S^1_1(k, \epsilon)D^1y^\Gamma(k, \epsilon).
\end{equation*}
In view of
\begin{equation*}
\begin{split}
S^1_1(k, 0)D^1y^\Gamma(k, \varrho)-S^1_1(k, \epsilon)D^1y^\Gamma(k, \epsilon)&=\qnt{S^1_1(k, 0)-S^1_1(k, \epsilon)}D^1y^\Gamma(k, \varrho)\\
&\quad +S^1_1(k, \epsilon)\qnt{D^1y^\Gamma(k, \varrho)-D^1y^\Gamma(k, \epsilon)},
\end{split}
\end{equation*}
as the deduction of \eqref{eq:W1Estimate}, we have
\begin{equation}\label{eq:S1Estimate}
\abs{S^1_1(k, 0)\qnt{D^1y^\Gamma(k, \varrho)-D^1\Upsilon(k)}}\leq C\epsilon(\sharp-k)^{-\alpha-1}.
\end{equation}
Combine \eqref{eq:W1Estimate} and \eqref{eq:S1Estimate}, we have \eqref{eq:PerturbationEstimate01} for $i=1$.

To estimate $\nabla^2 y^\Gamma(k, \varrho)-\nabla^2 \Upsilon(k)$, differentiating \eqref{eq:yGammaB1}, we have
\begin{equation}\label{eq:yGammaB2}
y^\Gamma_{kk}(k, \Gamma(k))=B''(k)-y^\Gamma_{k}(k, \Gamma(k))\Gamma''(k)-2y^\Gamma_{k\varrho}(k, \Gamma(k))\Gamma'(k)-y^\Gamma_{\varrho\varrho}(k, \Gamma(k))\qnt{\Gamma'(k)}^2.
\end{equation}
Since
\begin{equation*}
y^\Gamma_{kk}=y^\Gamma_{uu}u_k^2+2y^\Gamma_{uv}u_kv_k+y^\Gamma_{vv}v_k^2+y^\Gamma_uu_{kk}+y^\Gamma_vv_{kk}.
\end{equation*}

We set
\begin{equation*}
\begin{cases}
W_2^2(l, \varsigma)D^2y(k, \varrho)=u_k^2(l, \varsigma)y_{uu}(k, \varrho)+2u_kv_k(l, \varsigma)y_{uv}(k, \varrho)+v_k^2(l, \varsigma)y_{vv}(k, \varrho)\\
W_1^2(l, \varsigma)D^1y(k, \varrho)=u_{kk}(l, \varsigma)y_u(k, \varrho)+v_{kk}(l, \varsigma)y_v(k, \varrho)
\end{cases}
\end{equation*}
\eqref{eq:yGammaB2} can be rewritten as
\begin{equation*}
\begin{split}
&W_2^2(k, \Gamma(k))D^2y^\Gamma(k, \Gamma(k))+W_1^2(k, \Gamma(k))D^1y^\Gamma(k, \Gamma(k))\\
&\quad=B''(k)-y^\Gamma_{k}(k, \Gamma(k))\Gamma''(k)-2y^\Gamma_{k\varrho}(k, \Gamma(k))\Gamma'(k)-y^\Gamma_{\varrho\varrho}(k, \Gamma(k))\qnt{\Gamma'(k)}^2,
\end{split}
\end{equation*}
and we have the following expression
\begin{equation}\label{eq:yGammaB2W}
\begin{split}
&W_2^2(k, 0)D^2y^\Gamma(k, \varrho)+W_1^2(k, 0)D^1y^\Gamma(k, \varrho)\\
&\quad=\qnt{W_2^2(k, 0)D^2y^\Gamma(k, \varrho)+W_1^2(k, 0)D^1y^\Gamma(k, \varrho)}\\
&\quad\quad-\qnt{W_2^2(k, \Gamma(k))D^2y^\Gamma(k, \Gamma(k))+W_1^2(k, \Gamma(k))D^1y^\Gamma(k, \Gamma(k))}\\
&\quad\quad+B''(k)-y^\Gamma_{k}(k, \Gamma(k))\Gamma''(k)-2y^\Gamma_{k\varrho}(k, \Gamma(k))\Gamma'(k)-y^\Gamma_{\varrho\varrho}(k, \Gamma(k))\qnt{\Gamma'(k)}^2.
\end{split}
\end{equation}

For $\Upsilon$, \eqref{eq:WallConditionBS2} is equivalent to
\begin{equation}\label{eq:UpsilonB2}
\Upsilon_{uu}(k)u_k^2(k, 0)+2\Upsilon_{uv}(k)u_kv_k(k, 0)+\Upsilon_{vv}(k)v_k^2(k, 0)+\Upsilon_u(k)u_{kk}(k, 0)+\Upsilon_v(k)v_{kk}(k, 0)=B''(k).
\end{equation}
Above can be rewritten as
\begin{equation}\label{eq:UpsilonB2W}
W_2^2(k, 0)D^2\Upsilon(k)+W_1^2(k, 0)D^1\Upsilon(k)=B''(k).
\end{equation}
Combining \eqref{eq:yGammaB2W} and \eqref{eq:UpsilonB2W}, we have
\begin{equation*}
\begin{split}
&\abs{W_2^2(k, 0)D^2\qnt{y^\Gamma(k, \varrho)-\Upsilon(k)}+W_1^2(k, 0)D^1\qnt{y^\Gamma(k, \varrho)-\Upsilon(k)}}\\
&\quad\leq C\epsilon C\abs{\sharp-k}^{-\alpha-2}+C\abs{\Gamma-\epsilon}^\beta C\abs{\sharp-k}^{-\alpha-\beta-2}+C\epsilon \abs{\sharp-k}^{-\alpha-2}.
\end{split}
\end{equation*}
Adopting \eqref{eq:PerturbationEstimate01} to estimate the first order term in above, we have
\begin{equation}\label{eq:W2Estimate}
\abs{W_2^2(k, 0)D^2\qnt{y^\Gamma(k, \varrho)-\Upsilon(k)}}\leq C\epsilon^\beta\abs{\sharp-k}^{-\alpha-2}.
\end{equation}

Next, in view of \eqref{eq:ySecondOrderEquation2}, we set
\begin{equation*}
\begin{cases}
C_2^2(l, \varsigma)D^2y(k, \varrho)=(c^2-v^2)(l, \varsigma)y_{uu}(k, \varrho)+(2uv)(l, \varsigma)y_{uv}(k, \varrho)+(c^2-u^2)(l, \varsigma)y_{vv}(k, \varrho);\\
C_1^2(l, \varsigma)D^1y(k, \varrho)=C_1(l, \varsigma)y_u(k, \varrho)+C_2(l, \varsigma)y_v(k, \varrho).
\end{cases}
\end{equation*}
Then \eqref{eq:ySecondOrderEquation2} is
\begin{equation*}
C_2^2(k, \varrho)D^2y^\Gamma(k, \varrho)+C_1^2(k, \varrho)D^1y^\Gamma(k, \varrho)=0,
\end{equation*}
and \eqref{eq:y2S} yields
\begin{equation*}
C_2^2(k, 0)D^2\Upsilon(k, 0)+C_1^2(k, 0)D^1\Upsilon(k, 0)=0.
\end{equation*}
As \eqref{eq:W2Estimate}, we have
\begin{equation}\label{eq:C2Estimate}
\abs{C_2^2(k, 0)D^2\qnt{y^\Gamma(k, \varrho)-\Upsilon(k)}}\leq C\epsilon^\beta\abs{\sharp-k}^{-\alpha-2}.
\end{equation}

Finally, differentiating \eqref{eq:yGammaS1} with respect to $k$, we have
\begin{equation}\label{eq:yGammaS2}
I_k(k, \epsilon; \epsilon)y_u^\Gamma+J_k(k, \epsilon; \epsilon)y_v^\Gamma+I(k, \epsilon; \epsilon)\qnt{y_{uu}^\Gamma u_k+y_{uv}^\Gamma v_k}+J(k, \epsilon; \epsilon)\qnt{y_{uv}^\Gamma u_k+y_{vv}^\Gamma v_k}=0
\end{equation}
Meanwhile, for \eqref{eq:ShockPolarObliqueConditionS2}, we have
\begin{equation}\label{eq:UpsilonS2}
I_k(k, 0; 0)\Upsilon_u+J_k(k, 0; 0)\Upsilon_v+I(k, 0; 0)\qnt{\Upsilon_{uu} u_k+\Upsilon_{uv} v_k}+J(k, 0; 0)\qnt{\Upsilon_{uv} u_k+\Upsilon_{vv} v_k}=0
\end{equation}
Set
\begin{equation*}
\begin{cases}
S_2^2(l, \varsigma)D^2y(k, \varrho)=I(l, \varsigma; \varsigma)u_k(l, \varsigma)y_{uu}(k, \varrho)  + J(l, \varsigma; \varsigma)v_k(l, \varsigma)y_{vv}(k, \varrho)\\
\quad\quad\quad\quad\quad\quad\quad\quad\quad+ \qnt{I(l, \varsigma; \varsigma)v_k(l, \varsigma) + J(l, \varsigma; \varsigma)u_k(l, \varsigma)}y_{uv}(k, \varrho),\\
S_1^2(l, \varsigma)D^1y(k, \varrho)=I_k(l, \varsigma; \varsigma)y_u(k, \varrho)+J_k(l, \varsigma; \varsigma)y_v(k, \varrho).
\end{cases}
\end{equation*}
\eqref{eq:yGammaS2} and \eqref{eq:UpsilonS2} become
\begin{equation*}
S_2^2(k, \epsilon)D^2y^\Gamma(k, \epsilon)+S_1^2(k, \epsilon)D^1y^\Gamma(k, \epsilon)=0,
\end{equation*}
and
\begin{equation*}
S_2^2(k, 0)D^2\Upsilon(k)+S_1^2(k, 0)D^1\Upsilon(k)=0.
\end{equation*}
Thus
\begin{equation*}
\begin{split}
&S_2^2(k, 0)D^2\qnt{y^\Gamma(k, \varrho)-\Upsilon(k)}+S_1^2(k, 0)D^1\qnt{y^\Gamma(k, \varrho)-\Upsilon(k)}\\
&\quad=S_2^2(k, 0)D^2y^\Gamma(k, \varrho)+S_1^2(k, 0)D^1y^\Gamma(k, \varrho)-\qnt{S_2^2(k, \epsilon)D^2y^\Gamma(k, \epsilon)+S_1^2(k, \epsilon)D^1y^\Gamma(k, \epsilon)}.
\end{split}
\end{equation*}
As \eqref{eq:W2Estimate}, we have
\begin{equation}\label{eq:S2Estimate}
\abs{S_2^2(k, 0)D^2\qnt{y^\Gamma(k, \varrho)-\Upsilon(k)}}\leq C\epsilon^\beta\abs{\sharp-k}^{-\alpha-2}.
\end{equation}

From \eqref{eq:W2Estimate}, \eqref{eq:C2Estimate}, and \eqref{eq:S2Estimate}, we derive \eqref{eq:PerturbationEstimate2}. The proof is complete.
\end{proof}

\begin{proof}[Proof of main theorem]
Especially, we have
\begin{equation*}
\abs{\partial_\varrho F(\frac{y_u}{y_v}\qnt{k, \varrho; \epsilon}, k, \varrho)\Big|_{\nabla y=\nabla y^\Gamma; \nabla^2 y=\nabla^2y^\Gamma}-\partial_\varrho F(\frac{y_u}{y_v}\qnt{k, \varrho; \epsilon}, k, \varrho)\Big|_{\varrho=\epsilon=0; \nabla y=\nabla\Upsilon; \nabla^2 y=\nabla^2\Upsilon}}\leq C\epsilon^\beta.
\end{equation*}
Therefore
\begin{equation*}
\begin{split}
&\abs{\frac{\dif \Gamma}{\dif k}-F(\frac{y_u^\Gamma}{y_v^\Gamma}\qnt{k, \Gamma; \epsilon}, k, \Gamma)}\\
&\quad\quad=\bigg|\partial_\varrho F(\frac{y_u}{y_v}\qnt{k, \varrho; \epsilon}, k, \varrho)\Big|_{\varrho=\epsilon=0; \nabla y=\nabla\Upsilon; \nabla^2 y=\nabla^2\Upsilon}\qnt{\Gamma-\epsilon}- \frac{s_g' \cdot \nabla_U G}{G_\varrho \qnt{s_g' \cdot \nabla_U k}}\Big|_{\varrho=\epsilon}\\
&\quad\quad\quad-\qnt{\qnt{F(\frac{y_u^\Gamma}{y_v^\Gamma}\qnt{k, \Gamma; \epsilon}, k, \Gamma)-F(\frac{y_u^\Gamma}{y_v^\Gamma}\qnt{k, \epsilon; \epsilon}, k, \epsilon)}+F(\frac{y_u^\Gamma}{y_v^\Gamma}\qnt{k, \epsilon; \epsilon}, k, \epsilon)}\bigg|\\
&\quad\quad=\bigg|\partial_\varrho F(\frac{y_u}{y_v}\qnt{k, \varrho; \epsilon}, k, \varrho)\Big|_{\varrho=\epsilon=0; \nabla y=\nabla\Upsilon; \nabla^2 y=\nabla^2\Upsilon}\qnt{\Gamma-\epsilon}- \frac{s_g' \cdot \nabla_U G}{G_\varrho \qnt{s_g' \cdot \nabla_U k}}\Big|_{\varrho=\epsilon}\\
&\quad\quad\quad-\partial_\varrho F(\frac{y_u^\Gamma}{y_v^\Gamma}\qnt{k, \varrho; \epsilon}, k, \varrho)\Big|_{\Gamma<\varrho<\epsilon}\qnt{\Gamma-\epsilon}+\frac{s_g' \cdot \nabla_U G}{G_\varrho \qnt{s_g' \cdot \nabla_U k}}\Big|_{\varrho=\epsilon}\bigg|\\
&\quad\quad\leq\abs{\partial_\varrho F(\frac{y_u}{y_v}\qnt{k, \varrho; \epsilon}, k, \varrho)\Big|_{\varrho=\epsilon=0; \nabla y=\nabla\Upsilon; \nabla^2 y=\nabla^2\Upsilon}-\partial_\varrho F(\frac{y_u^\Gamma}{y_v^\Gamma}\qnt{k, \varrho; \epsilon}, k, \varrho)\Big|_{\Gamma<\varrho<\epsilon}}\qnt{\Gamma-\epsilon}\\
&\quad\quad\leq C\frac{\epsilon^\beta}{\sharp-k}\cdot\epsilon(\sharp-k)=C\epsilon^{\beta+1},
\end{split}
\end{equation*}
where we adopt Theorem \ref{thm:PerturbationEstimates} to obtain the last inequality. Come back to the $(u, v)$ coordinates, we get \eqref{eq:ApproximateBoundaryEstimate}. The proof is complete.
\end{proof}

\section*{Acknowledgments}
The author would like to thank Huang Genggeng for his selfless and patient discussions.
\bibliographystyle{plain}
\bibliography{hypersonicflowpastwedge20240611}

\end{document}